\documentclass[11pt,reqno]{elsarticle}

\usepackage{amsfonts}
\usepackage{graphicx}
\usepackage{amscd}
\usepackage{amsmath}
\usepackage{xcolor}
\usepackage{epsfig}
\usepackage{subfigure}
\usepackage[english]{babel}
\usepackage{algorithm,algpseudocode} 
\usepackage{listings} 
\usepackage{multicol}
\usepackage[colorlinks=true, allcolors=blue]{hyperref}
\usepackage{tikz, amsmath}
\usetikzlibrary{shapes,arrows}

\providecommand{\U}[1]{\protect\rule{.1in}{.1in}}
\providecommand{\U}[1]{\protect\rule{.1in}{.1in}}
\allowdisplaybreaks
\newtheorem{theorem}{Theorem}
\newtheorem{lemma}{Lemma}

\newenvironment{proof}{\paragraph{Proof}}{\hfill$\square$}

\newtheorem{corollary}{Corollary}

\newtheorem{example}{Example}

\newtheorem{remark}{Remark}

\numberwithin{equation}{section}

\numberwithin{table}{section}
\numberwithin{figure}{section}
\usepackage{tikz}
\usetikzlibrary{arrows.meta,positioning}
\usetikzlibrary{arrows}
\usetikzlibrary{fit,shapes.geometric}

\begin{document}
\begin{flushleft}
{\Large
\textbf{Probabilistic solutions of fractional differential and partial differential equations and their Monte Carlo simulations}
} \newline
\\
Tamer Oraby, Harrinson Arrubla, Erwin Suazo$^{*}$
\\
School of Mathematical and Statistical Sciences, The University of Texas Rio Grande Valley, \\
1201 W. University Dr., Edinburg, Texas, USA,
\\
 $^*$Corresponding author: erwin.suazo@utrgv.edu  
\end{flushleft}

\section*{Abstract}
\indent 
The work in this paper is four-fold. Firstly, we introduce an alternative approach to solve fractional ordinary differential equations as an expected value of a random time process. Using the latter, we present an interesting numerical approach based on Monte Carlo integration to simulate solutions of fractional ordinary and partial differential equations. Thirdly, we show that this approach allows us to find the fundamental solutions for fractional partial differential equations (PDEs), in which the fractional derivative in time is in the Caputo sense and the fractional in space one is in the Riesz-Feller sense. Lastly, using Riccati equation, we study families of fractional PDEs with variable coefficients which allow explicit solutions. Those solutions connect Lie symmetries to fractional PDEs. 


{\bf Keywords:} Caputo fractional derivative, Riesz-Feller fractional derivative, Riccati equation, Lie symmetries, Green functions, Monte Carlo Integration, Mittag-Leffler function.

\section{Introduction}
\label{Sec:Intro}
Although it was started in the second half of the eighteenth century by Leibniz, Newton and l'H\^{o}pital, \cite{Consiglio2019}, fractional calculus has received great attention in the last two decades. Many physical, biological and epidemiological models have found that fractional order models could perform at least as good as well as their integer counterparts, \cite{Islam2020, Almeida}. Integer order models are also appearing as special cases of the fractional order. That makes the dynamical behavior of those models richer and in some cases flexible. As advances are made in fractional calculus and fractional modeling, understanding of the physical interpretation of fractional derivatives is becoming clearer. A memory kernel with algebraic decay is the most common interpretation for the change in the dynamical behavior of the system \cite{Stynes2018,Diethelm2020,Dokuyucu2020}. Today fractional calculus is widely used in physical modeling. Examples include the nonexponential relaxation in dielectrics and ferromagnets \cite{Nigmatullin}, \cite{WeronKlauser}, the diffusion processes \cite{MetzlerBarkaiKlafter}, \cite{MetzlerKlafter} and the Hamiltonian Chaos \cite{Zaslavsky} and \cite{SaichevZaslavsky}.

Anomalous diffusion could be modeled using fractional order stochastic processes and their Fokker-Planck equations \cite{Klafter1987,Sokolov2005,Yanovsky2000}. Continuous-time random walk (CTRW) is one approach used to model anomalous diffusion \cite{Metzler2000a}. In particular, a CTRW with infinite-mean time to jump exhibits sub-diffusion behavior. The time to jump could be modeled by a heavy tail distribution with index $\beta$ such that $0<\beta<1$ leading to a mean square displacement that is of order $t^{\beta}$ depicting the short-range jump and so thus sub-diffusion. A long-range jump leads to super-diffusion, i.e., $\beta>1$, \cite{Metzler2014}.

A subordinator is a non-decreasing L\'evy process, i.e. a non-negative process with independent and stationary increments. It is used as a random time, or operational time, in defining time-changed-processes. Its density decays algebraically as $1/t^{\alpha+1}$ as $t\rightarrow\infty$, with $\alpha \in (0,1)$. A different type of time-change can be made by using the inverse subordinator (i.e the hitting time of a subordinator) which sometimes leads to sud-diffusion processes \cite{Orsingher2010,Orsingher2016}.


The Riccati equation has played an important role in finding explicit solutions for Fisher and Burgers equations, (see \cite{Eule2006} and \cite{Feng} and references therein). Also, similarity transformations and the solutions of Riccati and Ermakov systems have been extensively applied thanks to Lie groups and Lie algebras \cite{Bluman}, \cite{BlumanKumei}, \cite{BlumanAnco}, \cite{BlumanCole} and \cite{Olver}. In this work, we show that this approach allows us to find the fundamental solutions for fractional PDEs; the fractional derivative in time is in the Caputo sense and the fractional derivative in space one is in the Riesz-Feller sense. We establish a relationship between the coefficients through the Riccati equation; we study families of fractional PDEs with variable coefficients which allow explicit solutions and find those explicit solutions. Those solutions connect Lie symmetries to fractional PDEs. 

 Formulating solutions of fractional differential equations and partial differential equations as expected values with respect to heavy-tail or power law distributions could be enabled using Monte-Carlo integration methods and Sampling Importance Integration to evaluate them. The mean problem is in the number of simulations or random number generations that need to be done to guarantee convergence and small standard error.

In Section 1, we will review fundamental definitions and classical results needed from the classical theory of fractional differential equations. In Section 2, we present the first main result of this work, Lemma \ref{lem3}, which allow us to see a fractional functions and their fractional derivatives as Wright type transformations of some functions and their derivatives. They could be also interpreted as expected values of functions in a random time process. Also, in Section 2 we present Theorem 2.1 which allow us to solve fractional ordinary differential equations through solutions of regular ordinary differential equations. In Section 3, we derive fractional green functions for some important fractional partial differential equations, like diffusion, Schr\"{o}dinger, and wave equation. In Section 4, we derive green functions of fractional partial differential equations with variable coefficients with application to Fokker-Planck equations. In Section 5, we use the integral transform or the expected value interpretation of the solutions of fractional equations to carry out Monte Carlo simulations of their solution. 

\subsection{Preliminaries}
\label{Sec:Prelim}
In this section we give the required background of Caputo and Riesz-Feller fractional differentiation.
\subsubsection*{Caputo Derivative.}

Let $D^n$ be the Leibniz integer-order differential operator given by 
\begin{equation*}
D^{n}f = \dfrac{d^{n}f}{dt^{n}} = f^{(n)},
\end{equation*}
and let $J^n$ be an integration operator of integer order given by 
\begin{equation}  
J^{n}f(t) = \dfrac{1}{n-1!} \int_{0}^{t}(t - \tau )^{n-1} f(\tau)d\tau,
\label{eq:intop}
\end{equation}
where $n\in \mathbb{Z}^{+}$. Let us use $D=D^{1}$ for the first derivative. For fraction-order integrals, we use
\begin{equation}
J^{n-\beta} f(t) = \dfrac{1}{\Gamma(n-\beta)} \int_{0}^{t}(t - \tau)^{n-\beta-1} f(\tau) d\tau,
\end{equation}
where $n-1 < \beta \leq n$. Now, define the Caputo fractional differential operator $D_C^\beta$ to be 
\begin{equation*}
D_{C}^{\beta}f(t) = J^{n-\beta}D^nf(t),
\end{equation*}
where $n-1<\beta \leq n$, for $n \in \mathbb{N}$. 
\subsubsection*{Riemann-Liouville Derivative.}
The Riemann-Liouville fractional differential operator $D_{RL}^{\beta}$ is defined to be
\begin{equation*}
D_{RL}^{\beta}f(t) = D^{n}J^{n-\beta}f(t),
\end{equation*}
where $n-1<\beta\leq n$, for $n \in\mathbb{N}$. We will use $
\partial^{\beta}_{t} F := \dfrac{\partial^{\beta}F}{\partial t^{\beta}}$ and use $\partial_{t}F := \dfrac{\partial F}{\partial t}$.

The Riemann-Liouville fractional is related to the Caputo fractional derivative
through \cite{Podlubny99}:
\begin{equation*}
D_{RL}^{\beta} f(t) = D_{C}^{\beta} f(t) + \sum_{k=0}^{n-1}f^{(k)}(0) \dfrac{t^{k}}{k!}.
\end{equation*}
While we will not discuss the Riemann-Liouville fractional derivatives in
the paper, the results presented in this paper are valid for Riemann-Liouville fractional derivatives, when $f^{(k)}(0)=0$ for $k = 0, 1, \ldots, n-1$.

We will consider $n=1$ in this work; that is $0<\beta\leq 1$. Some of the
results be extended through the remark that for $0<\beta\leq 1$, $%
D_C^{n+\beta-1} f(t)=D_C^{\beta} f^{(n-1)}(t)$ for $n\geq 1$. Note also that when $0<\beta\leq 1$, $D_{RL}^{\beta} f(t)=D_{C}^{\beta} f(t) + f(0)$.
\subsubsection*{Riesz-Feller Derivative}

The Riesz-Feller fractional differential operator $%
D_{RF}^{\alpha,\theta}$ is defined to be \cite{MainYuriPagni2001} 
\begin{eqnarray*}
D_{RF}^{\alpha ,\theta }f(x) &=&\frac{\Gamma (1+\alpha )}{\pi }\sin ((\alpha
+\theta )\frac{\pi }{2})\int_{0}^{\infty }\dfrac{f(x+y)-f(x)}{y^{1+\alpha }}%
dy \\
&&+\frac{\Gamma (1+\alpha )}{\pi }\sin ((\alpha -\theta )\frac{\pi }{2}%
)\int_{0}^{\infty }\dfrac{f(x-y)-f(x)}{y^{1+\alpha }}dy
\end{eqnarray*}%
for fractional order $0<\alpha \leq 2$, and the skewness parameter $\theta\leq
\min(\alpha,2-\alpha)$. The symmetric Riesz-Feller differential operator is
defined at $\theta=0$ and is simply denoted by $D_{RF}^{\alpha} $.
\subsubsection*{Transformations.}
The Laplace transform of a function $f(t)$ is defined as 
\begin{equation*}
\mathcal{L}(f)(s) = \widetilde{f}(s) = \int_{0}^{\infty} e^{-st}f(t)dt.
\end{equation*}
The inverse Laplace transform is defined by 
\begin{equation*}
\mathcal{L}^{-1}\left(\widetilde{f}\right)(t) = \dfrac{1}{2\pi i}\int_{\mathcal{C}} e^{st} \widetilde{f}(s)ds
\end{equation*}
where $\mathcal{C}$ is a contour parallel to the imaginary axis and to the right of the singularities of $\widetilde{f}$. The Laplace transform of the Caputo fractional derivative of a function is given by 
\begin{equation}  
\label{eq:LC}
\mathcal{L}\left(D_{C}^{\beta}f\right)(s) = s^{\beta} \widetilde{f}(s) - \sum_{k=0}^{n-1}
s^{\beta-1-k}f^{(k)}(0).
\end{equation}
The Fourier transform of a function $f(x)$ is defined as
\begin{equation*}
\mathcal{F}(f)(y) = \widehat{f}(y) = \int_{-\infty}^{\infty} e^{ixy}f(x)dx.
\end{equation*}
The inverse Fourier transform is defined by
\begin{equation*}
\mathcal{F}^{-1}\left(\widehat{f} \right)(x)=\dfrac{1}{2 \pi}\int_{-\infty}^\infty e^{-ixy}\widehat{f}(y)dy.
\end{equation*}
The Fourier transform of the Riesz-Feller fractional derivative of a function is given by
\begin{equation}  \label{eq:FRF}
\mathcal{F}\left(D_{RF}^{\alpha,\theta} f\right)(y) =  - \psi_{\alpha}^{\theta}(y)\widehat{f}(y),
\end{equation}
where $\psi_{\alpha}^{\theta}(y) = |y|^{\alpha} e^{\frac{i \text{sign}(y)\theta\pi}{2}}$.
\subsubsection*{Mittag-Leffler Function}
The Mittag-Leffler function, which generalizes the exponential function, can be written as follows, 
\begin{equation}
\begin{split}
    E_{\beta} (z) = \sum\limits_{k=0}^{\infty} \dfrac{z^k}{\Gamma(\beta k+1)},
\end{split} \qquad
\begin{split}
    \beta \in \mathbb{R}^{+}, \enspace z \in \mathbb{C},
\end{split}
\end{equation}
and the more general Mittag-Leffler function with two-parameters is defined to be 
\begin{equation}
\begin{split}
    E_{\beta,\alpha} (z) = \sum\limits_{k=0}^{\infty} \dfrac{z^k}{\Gamma(\beta
k+\alpha)},
\end{split}\qquad
\begin{split}
    \beta,\alpha \in \mathbb{R}^+, \enspace z \in \mathbb{C}.
\end{split}
\end{equation}

\subsubsection*{Wright function.}
The Wright function is another special function of importance to fractional calculus and is defined by \cite{MainGloVivo2001},
\begin{equation}  
\label{eq:wf}
\begin{split}
    W_{\beta,\alpha} (z) = \sum\limits_{k=0}^{\infty} \dfrac{z^k}{k!\Gamma(\beta k+\alpha)},
\end{split}\qquad
\begin{split}
    \beta>-1,\alpha \in \mathbb{C}, \enspace %
z \in \mathbb{C}.
\end{split}
\end{equation}
The following Wright type function will be fundamental in the rest of this work
\begin{equation}
    g_{\beta}(x;t) = \dfrac{1}{t^{\beta}}W_{-\beta,1-\beta}\left(-\dfrac{x}{t^{\beta}}\right),
\end{equation}
 which is a probability density function of the random time process $\mathcal{T}_{\beta}(t)$ for all $t>0$ ~\cite{MainConsi2020,Gorenflo2010}. $\mathcal{T}_{\beta}(t)$ can be seen as the inverse of a $\beta$ stable subordinator with density $g_{\beta}(x,t)$, see \cite{Bingham1971, Meerschaert2012, Meerschaert2013}.
 It has a Laplace transform 
\begin{equation}  
\label{eq:lw}
\mathcal{L}\left(g_{\beta}(\cdot;t)\right)(s) = \int_{0}^{\infty}e^{-sx} g_{\beta}(x;t)dx = E_{\beta}\left(-st^{\beta}\right)
\end{equation}
for $\Re(s)>0$ and moments $\mathbb{E}\left[(\mathcal{T}_{\beta}(t))^{k}\right] =  \Gamma(k+1) \dfrac{t^{k\beta}}{\Gamma(k\beta+1)}$ for $k\geq1$~\cite{Piryatinska2005,Kumar2018}. At the same time     
\begin{equation}  
\label{eq:lw2}
\int_{0}^{\infty} e^{-st} g_{\beta}(x;t)dt = s^{\beta-1} e^{-xs^\beta}.
\end{equation}
It is shown in~\cite{MainConsi2020} that for $0<\beta,\alpha<1$, 
\begin{equation}  
\label{eq:CK}
g_{\beta \alpha}(x;t) = \int_{0}^{\infty} g_{\beta}(x;s) g_{\alpha}(s;t)ds
\end{equation}
for $x>0$. For more details about Wright function see~\cite{MainConsi2020,Gorenflo2010}. From \eqref{eq:lw} and \eqref{eq:CK}, we get
\begin{equation*}
\int_{0}^{\infty}E_{\beta}\left(-st^{\beta}\right) g_{\alpha}(t;r)dt = \int_{0}^{\infty}\int_{0}^{\infty}e^{-sx} g_{\beta}(x;t) g_{\alpha}(t;r)dxdt = E_{\beta \alpha}\left(-sr^{\beta \alpha}\right),
\end{equation*}
and from \eqref{eq:lw2} and \eqref{eq:CK}, we get 
\begin{equation*}
\int_{0}^{\infty} E_{\alpha}\left(-tx^{\alpha}\right) g_{\beta}(x;t)dt = \int_{0}^{\infty}s^{\beta-1} e^{-xs^\beta}g_{\alpha}(s;x)ds.
\end{equation*}
\subsubsection*{L\'evy $\alpha$-stable distribution.}
The L\'evy $\alpha$-stable distribution with stability index $0< \alpha \leq 2$, $L_{\alpha}^{\theta}(x)$ has a Fourier transform given by 
\begin{equation}\label{levy-alpha}
\mathcal{F}\left(L_{\alpha}^{\theta}(\cdot)\right)(y) := \widehat{L}_{\alpha}^{\theta}(y) = e^{-\psi_{\alpha}^{\theta}(y)}.
\end{equation}
The density $L_{\alpha}^{\theta}(x)$ has a fat tail proportional to $\lvert x\rvert^{-(1+\alpha)}$.

Define $L_{\alpha}^{\theta}(y;x) = \dfrac{1}{x^{\frac{1}{\alpha}}} L_{\alpha}^{\theta}\left(\dfrac{y}{x^{\frac{1}{\alpha}}}\right)$ for $y\in \mathbb{R}$ and $x>0$ which is a probability density function of the $\alpha$-stable random process with asymmetry parameter $\theta$, denoted by $\mathcal{L}_{\alpha}^{\theta}(x)$
for $x>0$, \cite{MainardiPagni2008}.
\begin{equation}  
\label{eq:fw1}
\int_{-\infty}^{\infty} e^{isy} L_{\alpha}^{\theta}(y;x)dy = e^{-\psi_{\alpha}^{\theta}(s) x}.
\end{equation}
For $0<\alpha<1$, and $t,x>0$, 
\begin{equation}
L_{\alpha}^{-\alpha}(t;x) = \dfrac{x \alpha}{t} g_{\alpha}(x;t).
\end{equation}
See \cite{MainMurPag2009} for more details. Also, for $0<\alpha \leq 1$ 
\begin{equation}  
\label{eq:CL}
L_{\beta \alpha}^{\theta \alpha}(x;t) = \int_{0}^{\infty}
L_{\beta}^{\theta}(x;s) L_{\alpha}^{-\alpha}(s;t) ds.
\end{equation}
\section{Fractional Derivative as Expected Value with Respect to a L\'evy Distribution}
\label{Sec:MandM}
In this Section we establish the first main result of this paper.
\subsection{Caputo Fractional Derivative}
The following lemma is one of the main results of this work for its wide applicability. 
\begin{lemma}\label{lem3} 
\leavevmode
\begin{enumerate}
\item Let $f\in C([0,\infty ))$ and $f_{\beta }$ be a function.
Hence, 
\begin{equation}\label{transformation}
\widetilde{f}_{\beta }(s)=s^{\beta -1}\widetilde{f}\left(s^{\beta}\right)
\end{equation}
for $s\in [0,\infty)$ if and only if
\begin{equation*}
f_{\beta }(t)=\int_{0}^{\infty }f(x)g_{\beta }(x;t)dx
\end{equation*}
for $0<\beta \leq 1$, and if integrals exist.

\item Let $f\in C^{1}([0,\infty ))$ and $f_{\beta }$ be a function such that 
\begin{equation*}
\widetilde{f}_{\beta }(s) = s^{\beta -1}\widetilde{f}\left(s^{\beta}\right)
\end{equation*}
for $s\in [0,\infty)$ and $f_{\beta }(0)=f(0)$ then
\begin{equation*}
D_{C}^{\beta }f_{\beta }(t)=\int_{0}^{\infty }f^{(1)}(x)g_{\beta }(x;t)dx
\label{eq:Clm}
\end{equation*}
for $0<\beta \leq 1$.

\item Let $f\in C^{n}([0,\infty))$  such that 
\begin{equation}\label{transformation21}
\mathcal{L}\left(f_{\beta }^{(n-1)}(\cdot)\right)(s) = s^{\beta -1} \mathcal{L}\left(f^{(n-1)}(\cdot)\right)\left(s^{\beta }\right)
\end{equation}
for $s\in [0,\infty)$ and $f_{\beta }^{(n-1)}(0)=f^{(n-1)}(0)$. For $0<\beta \leq 1$, the following holds
\begin{equation*}
D_{C}^{n+\beta -1}f_\beta(t)=\int_{0}^{\infty }f^{(n)}(x)g_{\beta }(x;t)dx
\label{NDerivative}
\end{equation*}
for $n\geq1$.

\item For $0<\alpha,\beta < 1$ and such that $f_{\beta \alpha}(0) = f_{\beta}(0)$
\begin{equation*}
D_{C}^{\beta \alpha }f_{\beta \alpha}(t) = D_{C}^{\alpha }\left(D_{C}^{\beta}f_{\beta}(s)\right).
\end{equation*}

\end{enumerate}
\end{lemma}
\begin{remark}
The condition in equation \eqref{transformation21} of part (3), is equivalent to stating that $f_\beta^{(n-1)}=\left(f^{(n-1)}\right)_\beta$ in the sense of equation \eqref{transformation}.
\end{remark}
\begin{proof}
\leavevmode
\begin{enumerate}
\item To show sufficiency, we use the Laplace transform as follows:
\begin{align*}
\begin{split}
\mathcal{L}\left(\int_{0}^{\infty }f(x)g_{\beta}(x;\cdot)dx\right)(s) &= \int_{0}^{\infty} \int_{0}^{\infty} e^{-st}f(x)g_{\beta }(x;t)dxdt \\
\text{by equation \eqref{eq:lw2}} &= \int_{0}^{\infty}f(x)s^{\beta-1} e^{-xs^{\beta}}dx \\
&= s^{\beta-1} \widetilde{f}\left(s^{\beta}\right)  \\
&= \widetilde{f}_{\beta}(s) = \mathcal{L}\left(f_{\beta}(\cdot)\right)(s).   
\end{split}
\end{align*}
Necessity follows from the same lines.

\item We will show that the Laplace transform of the right-hand side of equation \eqref{eq:Clm} is equal to that of the left hand side which is given by \eqref{eq:LC}. The Laplace transform of the right hand side of equation \eqref{eq:Clm} is given by 
\begin{align*}
\begin{split}
\mathcal{L}\left(\int_{0}^{\infty}f^{(1)}(x) g_{\beta}(x;\cdot)dx\right)(s) &=  \int_{0}^{\infty}\int_{0}^{\infty}e^{-st}f^{(1)}(x)g_{\beta }(x;t)dxdt \\
\text{by equation \eqref{eq:lw2}} &= \int_{0}^{\infty} f^{(1)}(x)s^{\beta-1} e^{-xs^{\beta}}dx  \\
&= s^{\beta-1}\left(s^{\beta}\widetilde{f}(s^{\beta}) - f(0)\right) \\
&= s^{\beta}\widetilde{f}_{\beta}(s) - s^{\beta-1}f_{\beta}(0) \\ 
&= \mathcal{L}\left(D_{C}^{\beta}f_{\beta}(\cdot)\right)(s). 
\end{split}
\end{align*}

\item  It is easy to prove that condition \eqref{transformation} is equivalent to 
\begin{equation}\label{transformation2}
    f^{(n-1)}_{\beta}(t)=(f^{(n-1)})_{\beta}(t)
\end{equation}
   
By properties of Caputo derivative, conditions \eqref{transformation2}  and $f_{\beta }^{(n-1)}(0)=f^{(n-1)}(0)$ we obtain  

\begin{equation*}\label{transformation3}
D_{C}^{\beta+n-1 }f_{\beta}(t)=D_{C}^{\beta }f^{(n-1)}_{\beta}(t)=D_{C}^{\beta }(f^{(n-1)})_{\beta}(t)=\int_{0}^{\infty }f^{(n)}(x)g_{\beta }(x;t)dx.
\end{equation*}

as we wanted.

\item Using equation \eqref{eq:CK}, 
\begin{equation*}
D_{C}^{\beta \alpha}f_{\beta \alpha}(t)=\int_{0}^{\infty}f'(x) g_{\beta\alpha}(x;t)dx = \int_{0}^{\infty}f'(x) \int_{0}^{\infty} g_{\beta }(x;s)g_{\alpha}(s;t)dsdx,
\end{equation*}
and using part 2, we obtain 
\begin{align*}
\begin{split}
D_{C}^{\beta \alpha }f_{\beta \alpha}(t) &= \int_{0}^{\infty }g_{\alpha
}(s;t)\int_{0}^{\infty }f^{\prime }(x)g_{\beta}(x;s)dxds \\ 
&= \int_{0}^{\infty}D_{C}^{\beta }f_{\beta }(s)g_{\alpha}(s;t)ds = D_{C}^{\alpha }\left(D_{C}^{\beta }f_{\beta }(t)\right). 
\end{split}
\end{align*}
\end{enumerate}
\end{proof}
\begin{theorem}\label{thm21}
\label{TheoremCaputo} The linear fractional differential equation 
\begin{equation}
\sum_{k=1}^{n}a_{k}D_{C}^{k+\beta -1}y(t)+a_{n+1}y(t)=F_{\beta }(t)
\label{eq:ODE}
\end{equation}%
such that $ y(0)=y_{0}$ and $D_{C}^{\beta+i-1} y(0)=y_{i}$ for $i=1,\ldots,n-1$ and $0<\beta <1$, has a solution given by 
\begin{equation*}
y_{\beta }(t)=\int_{0}^{\infty }z(x)g_{\beta }(x;t)dx
=\mathbb{E}\left[z(\mathcal{T}_{\beta}(t))\right]. \end{equation*}%
Where $z(t)$ is the solution of the linear ordinary differential equation 
\begin{equation}
\sum_{k=1}^{n}a_{k}z^{(k)}(t)+a_{n+1}z(t)=F(t)  \label{eq:FDE}
\end{equation}%
such that $z^{(k)}(0)=y_{i}$ for $k=0,\ldots,n-1$ and both solutions satisfy the following condition
\begin{equation*}
\mathcal{L}\left(y^{(n-1)}(\cdot)\right)(s) = s^{\beta -1} \mathcal{L}\left(z^{(n-1)}(\cdot)\right)\left(s^{\beta }\right)
\end{equation*}
for $s\in [0,\infty)$.
Where $F_{\beta }(t) = \int _{0}^{\infty }F(x)g_{\beta
}(x;t)dx$ and the process $\mathcal{T}_{\beta}(t)$ can be seen as the inverse of a $\beta$ stable subordinator whose density is $g_{\beta}(x,t)$. 
\end{theorem}
{\bf Remark.}
The function $F(x)$ could be found using Laplace transform and Lemma \ref{lem3} part a.
\begin{equation*}
\widetilde{F}_{\beta }(s)=s^{\beta -1}\widetilde{F}\left(s^{\beta}\right).
\end{equation*}

\begin{proof}
Taking a $\beta -$Wright type transformation on both sides of equation %
\eqref{eq:FDE}, we obtain%
\begin{equation*}
\sum_{k=1}^{n}a_{k}\int_{0}^{\infty }z^{(k)}(x)g_{\beta
}(x;t)dx+a_{n+1}\int_{0}^{\infty}z(x)g_{\beta }(x;t)dx=\int_{0}^{\infty
}F(x)g_{\beta }(x;t)dx,
\end{equation*}

\bigskip and applying Lemma \ref{lem3} we obtain 
\begin{equation*}
\sum_{k=1}^{n}a_{k}D_{C}^{k+\beta -1}y_{\beta }(t)+a_{n+1}D_{C}^{\beta
-1}y_{\beta }(t)=F_{\beta }(t)
\end{equation*}

as we wanted.
\end{proof}

In the following, let $0 < \beta \leq 1$. The best method to find 
\begin{equation*}
f_{\beta}(t)=\int_0^\infty f(x)g_{\beta}(x;t)dx
\end{equation*}
is by using Taylor's expansion of $f(t)$ about $0$ giving 
\begin{equation}
f_{\beta}(t)=\sum_{n=0}^\infty f^{(n)}(0) \dfrac{t^{\beta n}}{%
\Gamma(n\beta+1)}.
\label{eq:taylor}
\end{equation}
That relationship can be used to find many fractional analogue functions
that can be found in the literature. It will be also seen through the
following examples.

\begin{example}[Fractional velocity]
\label{exm:exm1}
The solution of the FDE 
$D_{C}^{\beta }y_{\beta }(t)=c$ with $y_{\beta }(0)=y_{0}$, where $c$ is a
real-valued constant is given by 
\begin{equation*}
   y_{\beta }(t)=\int_{0}^{\infty
}(y_{0}+cx)g_{\beta }(x;t)dx=y_{0}+\dfrac{ct^{\beta }}{\Gamma (\beta +1)} 
\end{equation*}
since $y(t)=y_{0}+ct$ solves $Dy(t)=c$, with $y(0)=y_{0}$, this solution agrees with the literature \cite{Jin2021}. 
That is, $y_{\beta}(t)=\mathbb{E}[y_0+c\mathcal{T}_{\beta }(t)]$. 
\end{example}
\begin{example}[Fractional growth/decay models]
\label{exm:exm2}
The solution of the FDE $D_C^{\beta} y(t)=\lambda y(t)$ with $y(0)=y_0$ where $\lambda$
is a real-valued constant is given by
\begin{equation*}
y(t)=\int_0^\infty y_0 e^{\lambda x}
g_{\beta}(x;t)dx= y_0 E_{\beta}\left(\lambda t^{\beta}\right)
\end{equation*}
since $z(x)=y_0 e^{\lambda x}$ solves $D z(x)=\lambda z(x)$, with $z(0)=y_0$, this solution agrees with the literature \cite{Jin2021}. That is, $y_{\beta}(t)=\mathbb{E}[y_0 e^{\lambda\mathcal{T}_{\beta }(t)}]$. See section \ref{numerical} for graphical representation. 
\end{example}
\begin{example}[Fractional oscillations]
\label{exm:exm3}
The solution of the FDE $D_C^{\beta+1} y(t)=-\omega^2 y(t)$ with $y(0)=0$ and $y'(0)=1$, where $\omega$ is a real-valued constant is given by 
\begin{equation*}
y(t)=\int_0^\infty \sin(\omega x)g_{\beta}(x;t)dx= \sin_{\beta}\left(\omega t\right) 
\end{equation*}
since $z(x)=\sin(\omega x) $ solves $D^2 z(x)=-\omega^{2} z(x)$, with $z(0)=0$. That is, $y_{\beta}(t)=\mathbb{E}[ \sin(\omega\mathcal{T}_{\beta }(t))]$. 

Using \eqref{eq:taylor} the fractional analogue of sine and cosine functions are given by 
\begin{equation*}
\sin_{\beta}\left(t\right) = \sum_{n=0}^\infty (-1)^{n} \dfrac{t^{(2n+1)\beta }}{\Gamma((2n+1)\beta +1)} = t^{\beta}E_{2\beta,\beta+1}\left(-t^{2\beta}\right)
\end{equation*}
and, 
\begin{equation*}
\cos_{\beta}\left(t\right) = \sum_{n=0}^\infty (-1)^n \dfrac{t^{2n\beta }}{%
\Gamma(2n\beta +1)} = E_{2\beta}\left(-t^{2\beta}\right),
\end{equation*}
which agrees with \cite{Stanislavsky2004}. It is not hard then to see that $D_C^{\beta} \cos_{\beta}\left(t\right) = -
\sin_{\beta}\left(t\right)$ and $D_C^{\beta} \sin_{\beta} \left(t\right) = \cos_{\beta}\left(t\right)$.
\end{example}

\section{Green functions for
fractional partial differential equations}

In the following we will use the notation $_{C}D_{t}^{\beta }$ and $_{RF}D_{x}^{\alpha }$ for Caputo and Riesz-Feller derivatives, respectively, to identify the variable with respect to which the derivatives are calculated. 

\subsection{The fractional diffusion equation}

In this example, using the approach of the previous section, we find the
Green function for the fractional diffusion equation
\begin{equation}
\label{Fractional Heat}
\left\{\begin{split}
_{C}D_{t}^{\beta }u(x,t) &=\, _{RF}D_{x}^{\alpha,\theta }u(x,t) \\
u(x,0) &= f(x)
\end{split}\qquad
\begin{split}
(x,t) &\in \mathbb{R}\times [0,\infty) \\
x &\in \mathbb{R}
\end{split}\right.
\end{equation}
with $0< \beta \leq 1$, $0<\alpha \leq 2$ and $u(\pm\infty,t)=0$, $t>0$. Applying the Fourier transform to $x$, we obtain
\begin{align*}
_{C}D_{t}^{\beta }\widehat{u}(k,t) &= -\psi _{\alpha }^{\theta }(k)\widehat{u}(k,t) \\
\widehat{u}(k,0) &= \widehat{f}(k).
\end{align*}
Solving this ordinary differential equation using Lemma \ref{lem3}, we obtain
\begin{equation*}
\widehat{u}(k,t)=\int_{0}^{\infty} \widehat{f}(k)e^{-\psi_{\alpha}^{\theta }(k)x} g_{\beta}(x;t)dx = \widehat{f}(k) E_{\beta} \left(-\psi_{\alpha}^{\theta}(k)t^{\beta}\right),
\end{equation*}
see~\hyperref[exm:exm2]{example 2.1}. Using the Inverse Fourier Theorem and the convolution theorem we get
\begin{equation*}
u(x,t) = \int_{-\infty}^{\infty} G_{\alpha,\beta}^{\theta}(x-y,t)f(y)dy
\end{equation*}
where the Green function is given by
\begin{equation*}
G_{\alpha,\beta}^{\theta}(x,t) = \dfrac{1}{2\pi} \int_{-\infty}^{\infty} e^{ikx}E_{\beta}\left(-\psi_{\alpha}^{\theta}(k) t^{\beta}\right)dk.
\end{equation*}
In the particular case of $\theta =0$,
\begin{equation}\label{frac_diff_soln}
G_{\alpha,\beta}^{0}(x,t) = \dfrac{1}{2\pi}\int_{-\infty}^{\infty} e^{ikx}E_{\beta}\left(-\lvert k\rvert^{\alpha}t^{\beta}\right)dk.
\end{equation}
For the case $\beta =1,\alpha =2$, we obtain the classical heat equation
\begin{equation*}
\left\{\begin{split}
\partial_{t}u(x,t) &= \partial_{x}^{2}u(x,t) \\
u(x,0) &= f(x).
\end{split}\qquad
\begin{split}
(x,t) &\in \mathbb{R}\times [0,\infty) \\
x &\in \mathbb{R}
\end{split}\right.
\end{equation*}
Since $E_{1}(\lambda t) = e^{\lambda t}$, we obtain for $x\in \mathbb{R}$ and $t>0$
\begin{equation*}
G_{1,2}^{0}(x,t) = \dfrac{1}{2\pi }\int_{-\infty}^{\infty} e^{ikx- tk^{2}}dk=\dfrac{1}{\sqrt{4\pi t}}e^{-\frac{x^{2}}{4 t}}
\end{equation*}
or $L_2^0(x,t)$, by using standard formulas from Fourier transform. This result agrees with the Green function of the standard heat equation, see for example~\cite{Suazo2011}.
\begin{corollary}
The solution of the fractional heat equation \eqref{Fractional Heat} could be written as
\begin{equation}\label{new}
    u_{\alpha,\beta}(x,t)=\int_{0}^\infty \int_{0}^\infty L_2^0(x,\tau)L_{\alpha/2}^{-\alpha/2}(\tau,s)g_\beta(s,t) d\tau ds,
\end{equation}
where $L_2^0(x,\tau)$ is the solution of the classical heat equation.
\end{corollary}
\begin{proof} As was shown in \cite{MainYuriPagni2001}, it could be shown that the solution in \eqref{frac_diff_soln} could be rewritten as
$$u_{\alpha,\beta}(x,t)=\int_{0}^\infty L_{\alpha}^{0}(x,s)g_\beta(s,t) ds,$$ but for $0<\alpha/2 \leq 1$ 
\begin{equation}  
\label{eq:CL1}
L_{\alpha}^{0}(x,s) = \int_{0}^{\infty}
L_{2}^{0}(x,\tau) L_{\alpha/2}^{-\alpha/2}(\tau,s) d\tau,
\end{equation}
by equation \eqref{eq:CL}. Then the result follows.
\end{proof}

Based on this corollary, $u_{\alpha,\beta}(x,t)=\mathbb{E} \left[ L_{2}^{0}(x,
\mathcal{L}_{\alpha/2}^{-\alpha/2}(\mathcal{T}_{\beta}(t))) \right]$ where $\mathcal{L}_{\alpha/2}^{-\alpha/2}(t)$  is the $\alpha/2-$L\'evy process with $\theta=-\alpha/2$ and $\mathcal{T}_{\beta}(t)$ is the random $\beta-$time process. See section \ref{numerical} for graphical representation. 

The moments of $X_{\alpha,\beta}(t)\sim u_{\alpha,\beta}(\cdot,t)$ are given by \cite{MainardiPagni2008},
\begin{equation}
\mathbb{E}(|X_{\alpha,\beta}(t)|^s) =   t^{\frac{s\beta}{\alpha}} \frac{\Gamma(1-\frac{s}{\alpha})\Gamma(1+\frac{s}{\alpha})\Gamma(1+s)}{\Gamma(1-\frac{s}{2})\Gamma(1+\frac{s}{2})\Gamma(1+\frac{s\beta}{\alpha})}  
\end{equation}
for $-\min(1,\alpha)<\mathcal{R}(s)<\alpha$. That formula shows different regimes of diffusion, subdiffusion and superdiffusion based on whether $\frac{\alpha}{2}=\beta$, $\frac{\alpha}{2}>\beta$, or $\frac{\alpha}{2}<\beta$, respectively. 

Moreover, Equation \eqref{new} shows that the $\alpha-\beta$ fractional process $\{X_{\alpha,\beta}(t),t\geq 0\}$ is equivalent in distribution to the subordinated process $\{B(\mathcal{L}_{\alpha/2}^{-\alpha/2}(\mathcal{T}_{\beta}(t))),t\geq 0\}$, where $\{B(t),t\geq 0\}$ is a Brownian motion. That relationship postulates that L\'evy flights are random dilation, with probability distribution $L_{\alpha/2}^{-\alpha/2}$, of the standard deviation or the time parameter in the Brownian motion. That dilation results in an expand in the range of possible displacement by magnitude beyond the regular tails of the standard Gaussian distribution. The process $\{B(\mathcal{L}_{\alpha/2}^{-\alpha/2}(t)),t\geq 0\}$ was introduced in \cite{Huff2002}, and used in \cite{BenoitMandelbrotandHowardM.Taylor2018} to model stock price differences. See also more about those subordinated processes in \cite{Barndorff-Nielsen}.

\subsection{Fractional wave equation}

If we consider d'Alembert solution $u(x,t)=(f(x-kt)+f(x+kt))/2$ of the wave
equation 
\begin{equation}
\frac{\partial ^{2}u(x,t)}{\partial t^{2}}=k^2\frac{\partial ^{2}u(x,t)}{%
\partial x^{2}},  \label{waveequation}
\end{equation}
where $f$ is a smooth enough function. And if we take the $\beta -$ Wright
type transformation on both sides of  equation \eqref{waveequation},  we
obtain 
\begin{equation*}
\int_{0}^{\infty }\dfrac{\partial ^{2}u(x,s)}{\partial s^{2}}g_{\beta
}(s;t)ds=k^{2}\int_{0}^{\infty }\dfrac{\partial ^{2}}{\partial x^{2}}%
u(x,s)g_{\beta }(s;t)ds=k^{2}\dfrac{\partial ^{2}}{\partial x^{2}}%
\int_{0}^{\infty }u(x,s)g_{\beta }(s;t)ds
\end{equation*}
by Lemma \eqref{lem3} we see that 
\begin{equation}
u_{\beta }(x,t)=1/2\int_{0}^{\infty }u(x,s)g_{\beta }(s;t)ds=\int_{0}^{\infty
}[f(x-ks)+f(x+ks)]g_{\beta }(s;t)ds=(f_{\beta }(x-kt)+f_{\beta }(x+kt))/2
\label{fractionalwave}
\end{equation}
is a solution of the fractional wave equation
\begin{equation*}
_{C}D_{t}^{\beta +1}u_{\beta }(x,t)=k^{2}\dfrac{\partial ^{2}u_{\beta }}{%
\partial x^{2}}(x,t).
\end{equation*}
We refer to the reader to \cite{Fujita1990}, \cite{Meerschaert2019} and \cite{Mainardibook2010}, for previous work on fractional d'Alambert solution.

\subsection{Examples of Green functions for fractional Schr\"{o}dinger equation.} We show in this section that we can apply Lemma \eqref{lem3} to Schr\"{o}dinger equation. Since the work of \cite{Bingham1971} the connection with inverse stable subordinators and Wright type functions was made for difussion equations. These examples might raise the interest of specialists inverse stable subordinators to find similar results for dispersion equations.

One of the classical one dimensional fractional Schr\"{o}dinger equations with Caputo derivative in time considered in the literature is given by
\begin{equation}
i{}_{0}D_{t}^{\beta } \phi(x,t)=k\frac{\partial ^{2}\phi(x,t)}{%
\partial x^{2}}+c\phi(x,t),  \label{Schrodingerequation}
\end{equation}
Given that when the skewness $\theta=0$, we have ${}_{x}D_{0}^{\alpha}\phi=\frac{\partial ^{2}\phi(x,t)}{%
\partial x^{2}}$, we consider a more general version  
\begin{equation*}
\left\{ \begin{aligned} \begin{split} i{}_{0}D_{t}^{\beta }\phi _{\beta
,\alpha }(x,t) & = -h{}_{x}D_{\theta}^{\alpha+1} \phi_{\beta ,\alpha }(x,t)
+c\phi _{\beta ,\alpha } \\ \phi _{\beta ,\alpha }(x,0)& = f(x)
\end{split}\qquad \begin{split} (x,t)& \in \mathbb{R}\times \lbrack
0,\infty) \\ x& \in \mathbb{R} \end{split} \end{aligned} \right.
\end{equation*}
Let $0<\beta ,\alpha \leq 1$. The Green function of the time and space-fractional
linear Schr\"odinger equation is as follows. The linear Schr\"{o}dinger Fourier transform is

\begin{equation*}
\left\{ \begin{aligned} {}_{0}D_{t}^{\beta }\widehat{\phi_{\beta
,\alpha}}(x,t) &=-i\widehat{\phi_{\beta
,\alpha}}(k,t)\left[h\psi^{\theta}_{\alpha+1}(k)+ic\right] \\
\widehat{\phi_{\beta ,\alpha}}(k,0) &= F(k) \end{aligned} \right.
\end{equation*}

Then, applying lemma 3, the solution is given by $\widehat{\phi _{\beta
,\alpha }}(k,t)=\int_{0}^{\infty }z(x)g_{\beta }(x;t)dx$ where the function $%
z(t)=F(k)\exp \left( -it\left[ h\psi _{\alpha +1}^{\theta }(k)+c\right]
\right) $ and solves the last Fractional PDE initial problem. Now, setting
the new solution we have that

\begin{equation*}
\widehat{\phi _{\beta ,\alpha }}(k,t)=F(k)E_{\beta }\left( -it^{\beta }\left[
h\psi _{\alpha +1}^{\theta }(k)+c\right] \right) .
\end{equation*}

Next, applying the Fourier transform inverse and convolution theorems lead
us to

\begin{equation*}
\phi(x,t)=\mathcal{F}^{-1}\left\{F(k)E_{\beta}\left(-it^{\beta}\left[%
h\psi^{\theta}_{\alpha+1}(k)+c\right]\right)\right\} =
\int_{-\infty}^{\infty }G(x-y,t)f(y)dy
\end{equation*}

where the Green function is given by

\begin{equation*}
G(x,t)=\dfrac{1}{2\pi }\int_{-\infty }^{\infty
}\exp(ikx)E_{\beta}\left(-it^{\beta}\left[h\psi^{\theta}_{\alpha+1}(k)+c%
\right]\right)dk
\end{equation*}

In the particular case $\theta = 0$, and $\beta=\alpha = 1$ it results to

\begin{equation*}
G(x,t)=\dfrac{1}{2\pi }\int_{-\infty }^{\infty }\exp (-ithk^{2}+ixk-itc).
\end{equation*}

And by complex integration see \cite{Suazo2009} we obtain

\begin{align*}
G(x,t) = \sqrt{\dfrac{1}{4i\pi th}}\exp\left(\dfrac{-x^{2}}{4ith} -
itc\right).
\end{align*}

Similarly, we can consider the fractional Schr\"{o}dinger equation given by  
\begin{equation}
i{}_{0}D_{t}^{\beta } \phi(x,t)=k\frac{\partial ^{2}\phi(x,t)}{%
\partial x^{2}}+\frac{\partial \phi(x,t)}{%
\partial x},  \label{Schrodingerequation}
\end{equation}

we find the Green function for the more general fractional Schr\"{o}dinger equation 

\begin{equation*}
\left\{ \begin{aligned} \begin{split} {}_{0}D_{t}^{\beta }\phi _{\beta
,\alpha }(x,t) & =ih{}_{x}D_{\theta}^{\alpha+1} \phi_{\beta ,\alpha }(x,t) -
ic{}_{x}D_{\theta}^{\alpha}\phi _{\beta ,\alpha} \\ \phi _{\beta ,\alpha
}(x,0)& = f(x) \end{split}\qquad \begin{split} (x,t)& \in \mathbb{R}\times
\lbrack 0,\infty) \\ x& \in \mathbb{R} \end{split} \end{aligned} \right.
\end{equation*}

After taking the Fourier transform we obtain

\begin{equation*}
\left\{ \begin{aligned} {}_{0}D_{t}^{\beta}\widehat{\phi_{\beta
,\alpha}}(x,t) & =-i\widehat{\phi_{\beta
,\alpha}}(k,t)\left[h\psi^{\theta}_{\alpha+1}(k)-c\psi^{\theta}_{\alpha}(k)%
\right] \\ \widehat{\phi_{\beta ,\alpha}}(k,0)& = F(k) \end{aligned} \right.
\end{equation*}

Then, applying lemma 3, the solution is given by $\widehat{\phi _{\beta
,\alpha }}(k,t)=\int_{0}^{\infty }z(x)g_{\beta }(x;t)dx,$where the function $%
z(x)=F(k)\exp \left( -ix\left[ h\psi _{\alpha +1}^{\theta }(k)-c\psi
_{\alpha }^{\theta }(k)\right] \right) $ and solves the last Fractional PDE
initial problem. Now, setting the new solution we have that

\begin{equation*}
\widehat{\phi _{\beta ,\alpha }}(k,t)=F(k)E_{\beta }\left( -it^{\beta }\left[
h\psi _{\alpha +1}^{\theta }(k)-c\psi _{\alpha }^{\theta }(k)\right] \right)
.
\end{equation*}

Next, applying the inverse Fourier transform and convolution theorems lead
us to

\begin{equation*}
\phi(x,t)=\mathcal{F}^{-1}\left\{F(k)E_{\beta}\left(-it^{\beta}\left[%
h\psi^{\theta}_{\alpha+1}(k)-c\psi^{\theta}_{\alpha}(k)\right]%
\right)\right\} = \int_{-\infty}^{\infty }G(x-y,t)f(y)dy
\end{equation*}

where the Green function is given by

\begin{equation*}
G(x,t)=\dfrac{1}{2\pi }\int_{-\infty }^{\infty
}\exp(ikx)E_{\beta}\left(-it^{\beta}\left[h\psi^{\theta}_{\alpha+1}(k)-c%
\psi^{\theta}_{\alpha}(k)\right]\right)dk
\end{equation*}

In the particular case $\theta = 0$, and $\beta=\alpha = 1$ it results to

\begin{equation*}
G(x,t)=\dfrac{1}{2\pi }\int_{-\infty }^{\infty }\exp \left(
-ithk^{2}+itc|k|+ixk\right) dk.
\end{equation*}

For the Green functions for the regular versions of Schr\"{o}dinger equations we refer the reader to \cite{Suazo2009}.
\section{Green functions for fractional PDEs with variable coefficients}
In this section we use Lie symmetry group methods and Lemma \ref{lem3} to
introduce families of fractional partial differential equations with variabe
coefficients exhibiting explicit solutions. More specifically, we show that
solutions for FDE with variable coefficients of the form 

\begin{equation}
_{C}D_{t}^{\beta }u_{\beta}(x,t)=\sigma (x)\dfrac{\partial u_{\beta}}{\partial x^2}(x,t)+\mu (x)\dfrac{\partial u_{\beta}}{\partial x}(x,t), 0<\beta \leq 1.
\label{VC diffusion}
\end{equation}
can be expressed as a Wright type transformation for PDEs that we define as 
\begin{equation}
u_{\beta }(x,t)=\int_{0}^{\infty }u(x,s)g_{\beta }(s,t)ds,
\label{Transformation PDE}
\end{equation}
where where $u(x,t)$ is the solution of the associated standard PDE 
\begin{equation}
u_{t}(x,t)=\sigma (x)u_{xx}(x,t)+\mu (x)u_{x}(x,t).
\label{PDE familly}
\end{equation}
\begin{theorem}
\label{TheoremVC} Let's consider the fractional partial differential
equation with a variable coefficient of the form \eqref{VC diffusion}. The
solution is given by the Wright-type transformation define by \eqref%
{Transformation PDE} of the analogous PDE of the form \eqref{PDE familly}
which admits a Green function. Furthermore, 
\begin{equation*}
\lim_{\beta \rightarrow 1^{-1}}u_{\beta }(x,t)=u(x,t).
\end{equation*}
\end{theorem}
\begin{proof}
Taking the $\beta-$ Wright type transformation on both sides of %
\eqref{VC diffusion}, we get 
\begin{align*}
\int_{0}^{\infty} \dfrac{\partial u(x,s)}{\partial s} g_{\beta}(s;t) ds &=
\sigma(x)\int_{0}^{\infty} u_{xx}(x,s)g_{\beta}(s;t) ds \\
&+ \mu (x)\int_{0}^{\infty }u_{x}(x,s)g_{\beta}(s;t) ds.
\end{align*}
By Lemma \ref{lem3}, we obtain 
\begin{align*}
_{C}D_{t}^{\beta }u_{\beta }(x,t)& =\sigma (x)\dfrac{\partial ^{2}}{\partial
x^{2}}\int_{0}^{\infty }u(x,s)g_{\beta }(s;t)ds \\
& +\mu (x)\dfrac{\partial }{\partial x}\int_{0}^{\infty }u(x,s)g_{\beta
}(s;t)ds.
\end{align*}
Finally, we obtain 
\begin{equation*}
_{C}D_{t}^{\beta }u_{\beta }(x,t) = \sigma(x)\dfrac{\partial^{2}u_{%
\beta}(x,t)}{\partial x^{2}} + \mu(x)\dfrac{\partial u_{\beta}(x,t)}{%
\partial x}.
\end{equation*}
\bigskip For $0<v<1$, we recall that%
\begin{equation*}
\lim_{v\rightarrow 1^{-}}W_{-v,1-v}(-z)=\delta (x-1),
\end{equation*}
therefore, we get 
\begin{equation*}
\lim_{v\rightarrow 1^{-}}g_{v }(s;t)=\lim_{v\rightarrow 1^{-}}t^{-v
}W_{-v,1-v}(-st^{-v})=\frac{1}{t}\delta (st^{-1}-1)=\delta (s-t)
\end{equation*}
Therefore, if we assume $u(x,t)=\int_{0}^{\infty }G(x,y,t)f(y)dy,$ $u$ has a
Green function, using Fubini-Tonelli's theorem and assuming $f$ is a
suitable function, then 
\begin{eqnarray*}
\lim_{\beta \rightarrow 1^{-}}u_{\beta }(x,t) &=&\lim_{\beta \rightarrow
1^{-}}\int_{0}^{\infty }g_{\beta }(s;t)u(x,s)ds=\lim_{\beta \rightarrow
1^{-}}\int_{0}^{\infty }g_{\beta }(s;t)\int_{0}^{\infty }G(x,y,s)f(y)dyds \\
&=&\int_{0}^{\infty }\delta (s-t)G(x,y,s)f(y)dy=\int_{0}^{\infty
}G(x,y,t)f(y)dy.
\end{eqnarray*}
\end{proof}
Next, we will present the following families, exhibiting explicit solutions
thanks to the Theorem above. 
\subsection{Solutions for a fractional Fokker-Planck equation with a forcing
function.}
Let $\mathcal{T}_{\beta}(t)$ for all $t>0$ be a random time process  that is an inverse of a $\beta$ stable subordinator. Let it also be independent of a standard Brownian motion $W(t)$. The Langevin stochastic differential equation with continuous time change using $\mathcal{T}_{\beta}(t)$ is given by 
\begin{equation}
dX(t)=-f(X(t))d\mathcal{T}_{\beta}(t)+\sqrt{2D}dW(\mathcal{T}_{\beta}(t))  \label{SDE}
\end{equation}%
such that $X(0)=x_{0}$ and $f$ is a Lipschitz function, \cite{Hahn2012}.
The solution process $X(t)$ of the time-changed Langevin equation \eqref{SDE} is a
non-stationary stochastic process and is completely specified by finding the
probability density $p(x,t|x_{0})\geq 0$, which satisfies the fractional Fokker-Planck
equation 
\begin{eqnarray*}
_{C}D_{t}^{\beta }p &=&D\frac{\partial ^{2}p}{\partial x^{2}}+%
\frac{\partial }{\partial x}\left[ f(x)p\right]  \\
p(x,0|x_{0}) &=&\delta (x-x_{0}).
\end{eqnarray*}
Applying the Theorem of this Section, we obtain the following Corollaries.
\begin{corollary}
If $f$ satisfies the Riccati equation of the form 
\begin{equation*}
2f^{\prime }(x)-f^{2}(x)+\beta ^{2}x^{2}-\gamma +\dfrac{16v^{2}-1}{x^{2}}=0
\end{equation*}%
and $\beta ,$ $\gamma $ and $v$ are constants, then the family of the
fractional Fokker-Planck equation 
\begin{equation}
_{C}D_{t}^{\beta }u(x,t)=\dfrac{\partial ^{2}u}{\partial x^{2}}+\dfrac{%
\partial }{\partial x}\left[ f(x)u(x,t)\right]   \label{FFPE1}
\end{equation}%
where $f(x)=-f(-x)$ ($f$ is an odd function) admits an explicit solution of
the form \eqref{Transformation PDE} where $u$ is given by%
\begin{equation}
u(x,t)=\int_{0}^{\infty }p(x,t|y)\varphi (y)dy,  \label{FP Green}
\end{equation}
$p$ is given by 
\begin{eqnarray*}
p(x,t|x_{0}) &=&F\left( \frac{x}{\sqrt{4\sinh ^{2}(\beta t)}}\right) \frac{%
e^{\gamma t/4}}{\left( 4\sinh ^{2}(\beta t)\right) ^{1/4}} \\
&&\times \exp \left[ -\left( \frac{\beta }{4}\coth (\beta t)x^{2}+\dfrac{1}{2%
}\int f(x)dx\right) +\frac{x_{0}^{2}\beta }{2(1-e^{2\beta t})}\right] ,
\end{eqnarray*}%
and 
\begin{equation*}
F(z)=\left\{ 
\begin{array}{cc}
z^{1/2}\left[ A_{1}I_{2v}(kz)+A_{2}I_{-2v}(kz)\right] , & \text{for}\enspace %
z>0 \\ 
\lvert z\rvert ^{1/2}\left[ B_{1}K_{2v}(k\lvert z\rvert
)+B_{2}I_{2v}(k\lvert z\rvert )\right] , & \text{for}\enspace x<0%
\end{array}%
\right. 
\end{equation*}%
where $k=\beta x_{0}$ and $A_{1},$ $A_{2},$ $B_{1}$ and $B_{2}$ are
arbitrary constants be determined by boundary and continuity conditions and
if $\alpha $ is a real number 
\begin{eqnarray*}
I_{\alpha }(z) &=&\sum_{m=0}^{\infty }\frac{\left( z/2\right) ^{2m+\alpha }}{%
m!\Gamma (m+\alpha +1)}, \\
K_{\alpha }(z) &=&\frac{\pi }{2}\frac{I_{-\alpha }(z)-I_{\alpha }(z)}{\sin
(\alpha \pi )}.
\end{eqnarray*}
\begin{proof}
This Corollary is a direct consequence of the Theorem of this section and
the similarity solution presented in \cite{Bluman} using the method of Lie
group symmetries for the equation 
\begin{align}
\dfrac{\partial p}{\partial t} &= \dfrac{\partial^{2}p}{\partial x^{2}}+%
\dfrac{\partial}{\partial x}\left[f(x)p\right] \\
p(x,0|x_{0}) &= \phi(x-x_{0})
\end{align}
when $f$ is odd.
\end{proof}
\end{corollary}
\begin{corollary}
A particular case of interest is obtained if 
\begin{equation*}
f(x)=a x+\frac{b}{x}, \text{ such that }a >0\text{ and }-\infty <b <1.
\end{equation*}
The transition probability density is 
\begin{eqnarray*}
p(x,t|x_{0}) &=& a x_{0}^{1/2}\left( \frac{x}{x_{0}}\right) ^{-\frac{b }{2}%
}z^{\frac{1}{2}}I_{-\left( \frac{1}{2}+\frac{b }{2}\right) }(kz) \\
&&\times \frac{e^{\gamma t/4}}{\left( 4\sinh ^{2}(a t)\right) ^{1/4}}\exp
\left( -\left( \frac{a }{4}\coth (a t)\right) ^{2}-\dfrac{a x^{2}}{4}\right)
\end{eqnarray*}
for $x\geq 0$.
\end{corollary}
\begin{example}
Let's consider the standard fractional Fokker-Planck equation of the form 
\begin{align}
_{C}D_{t}^{\beta }u(x,t)& =\frac{\partial ^{2}u}{\partial x^{2}}+x\frac{%
\partial u}{\partial x}+u \\
u(x,0)& =f(x).
\end{align}%
By the Theorem of this Section, we obtain \eqref{Transformation PDE} where $%
p(x,t)$ satisfies the classical Fokker-Planck equation (FPE) 
\begin{equation*}
\dfrac{\partial p}{\partial t}=\dfrac{\partial ^{2}p}{\partial x^{2}}+x%
\dfrac{\partial p}{\partial x}+p.
\end{equation*}%
It is possible to find the Green function of FPE, see for example \cite%
{Suazo2011}. Therefore we obtain explicit expression for the solution 
\begin{equation*}
u(x,t)=\int_{0}^{\infty }\int_{0}^{\infty }\frac{\exp \left[ -\frac{\left(
x-e^{-s}y\right) ^{2}}{2\left( 1-e^{-2s}\right) }\right] }{\sqrt{2\pi \left(
1-e^{-2s}\right) }}g_{\beta }(s;t)dsf(y)dy.
\end{equation*}
\end{example}
\begin{corollary}
The family of fractional PDEs with space variable coefficients of the form 
\begin{align}
_{C}D_{t}^{\beta }u(x,t)& =xu_{xx}(x,t)+f(x)u_{x}(x,t)  \label{Family 1} \\
u(x,0)& =\varphi (x)
\end{align}%
where $f$ satisfies the Riccati equation of the form 
\begin{equation*}
xf^{\prime }-f+\frac{1}{2}f^{2}=Ax^{\frac{3}{2}}+Cx-\frac{3}{8}
\end{equation*}%
admits an explicit solution of the form \eqref{Transformation PDE} where $u$
is given by \eqref{FP Green} and $p(x,y,t)$ is the inverse Laplace transform
of 
\begin{align*}
U_{\lambda }(x,t)& =\sqrt{\frac{\sqrt{x}\left( 1+\lambda t\right) }{\sqrt{x}%
\left( 1+\lambda t\right) -A\lambda \frac{t^{3}}{12}}}\exp [S(x,y,t)] \\
& \times \exp \left[ -\frac{1}{2}\left( F(x)-F\left( \frac{(12(1+\lambda t)%
\sqrt{x}-A\lambda t^{3})^{2}}{144\left( 1+\lambda t\right) ^{4}}\right)
\right) \right] 
\end{align*}%
where $F^{\prime }(x)=\frac{f(x)}{x}$, 
\begin{equation*}
S(\lambda ,x,t)=-\dfrac{\lambda (x+Ct^{2}/2)}{1+\lambda t}-\dfrac{2At^{2}%
\sqrt{x}(3+\lambda t)}{3(1+\lambda t)^{2}}+\dfrac{A^{2}t^{4}(2\lambda
t(3+\lambda t/2)-3)}{108(1+\lambda t)^{3}}
\end{equation*}%
for $\lambda \geq 0.$
\end{corollary}
\begin{proof}
This Corollary is a direct consequence of the Theorem of this section and
the similarity solution presented in Theorem 6.1 on \cite{Craddock} for 
\begin{align}
\dfrac{\partial u}{\partial t} &= x\dfrac{\partial^{2}u}{\partial x^{2}} +
f(x) \dfrac{\partial u}{\partial x} \\
u(x,0) &= \varphi(x)
\end{align}
when $f$ is an odd function.
\end{proof}
\begin{corollary}
The family of fractional PDEs with space variable coefficients of the form 
\begin{align}
_{C}D_{t}^{\beta }u(x,t)& =xu_{xx}(x,t)+f(x)u_{x}(x,t) \\
u(x,0)& =\varphi (x)
\end{align}%
where $f$ satisfies the Riccati equation of the form 
\begin{equation*}
xf^{\prime }-f+\dfrac{1}{2}f^{2}=Ax+B
\end{equation*}%
admits an explicit solution of the form \eqref{Transformation PDE}  where $u$
is given by \eqref{FP Green} and $p(x,y.t)$ is the inverse Laplace transform
of 
\begin{equation*}
U_{\lambda }(x,t)=\exp \left[ -\dfrac{\lambda (x+At^{2}/2)}{1+\lambda t}-%
\dfrac{1}{2}\left( F(x)-F\left( \dfrac{x}{\left( 1+\lambda t\right) ^{2}}%
\right) \right) \right] 
\end{equation*}%
where $F^{\prime }(x)=\frac{f(x)}{x}$ and for $\lambda \geq 0$.
\end{corollary}
\begin{proof}
This Corollary is a direct consequence of the Theorem of this section and
the similarity solution presented in Theorem 4.1 on \cite{Craddock} for 
\begin{align}
\dfrac{\partial u}{\partial t} &= x\dfrac{\partial^{2}u}{\partial x^{2}} +
f(x) \dfrac{\partial u}{\partial x} \\
u(x,0) &= \varphi(x)
\end{align}
when $f$ is odd.
\end{proof}
\begin{example}
\begin{align}
_{C}D_{t}^{\beta }u(x,t)& =xu_{xx}(x,t)+\left( \dfrac{1+3\sqrt{x}}{2\left( 1+%
\sqrt{x}\right) }\right) u_{x}(x,t) \\
u(x,0)& =\varphi (x)
\end{align}%
has a solution of the form \eqref{Transformation PDE} where $%
u(x,t)=\int_{0}^{\infty }G(x,y,t)\varphi (y)dy$ and 
\begin{equation*}
G(x,y,t)=\dfrac{\cos \left( \dfrac{2\sqrt{xy}}{t}\right) }{\sqrt{\pi yt}%
\left( 1+\sqrt{x}\right) }\left( 1+\sqrt{y}\tanh \left( \dfrac{2\sqrt{xy}}{t}%
\right) \right) \exp \left[ -\dfrac{x+y}{t}\right] .
\end{equation*}
\end{example}
\begin{example}
\begin{align}
_{C}D_{t}^{\beta }u(x,t)& =xu_{xx}(x,t)+\left( \dfrac{1}{2}+\sqrt{x}\coth
\left( \sqrt{x}\right) \right) u_{x}(x,t) \\
u(x,0)& =\varphi (x)
\end{align}%
has a solution of the form \eqref{Transformation PDE}  where $%
u(x,t)=\int_{0}^{\infty }G(x,y,t)\varphi (y)dy$ and 
\begin{equation*}
G(x,y,t)=\dfrac{\sinh \left( \dfrac{2\sqrt{xy}}{t}\right) }{\sqrt{\pi yt}}%
\dfrac{\sinh \left( \sqrt{y}\right) }{\sinh \left( \sqrt{x}\right) }\exp %
\left[ -\dfrac{x+y}{t}-\dfrac{1}{4}t\right] .
\end{equation*}
\end{example}
\begin{theorem}
A solution for 
\begin{equation}  \label{compbetaalpha}
_{C}D_{t}^{\beta\alpha}u_{\beta\alpha}(x,t) = \sigma(x)\dfrac{\partial^{2}
u_{\beta\alpha}(x,t)}{\partial x^{2}} + \mu(x)\dfrac{\partial u_{\beta
\alpha}(x,t)}{\partial x}
\end{equation}
can be obtained as an $\alpha -$Wright type transformation of $u_{\beta}$,
where $u_{\beta}$ is a solution for 
\begin{equation*}
_{C}D_{t}^{\beta}u_{\beta }(x,t) = \sigma(x)\dfrac{\partial^{2}
u_{\beta}(x,t)}{\partial x^{2}} + \mu(x) \dfrac{\partial u_{\beta}(x,t)}{%
\partial x}.
\end{equation*}
\end{theorem}
\begin{proof}
Let's consider 
\begin{equation*}
_{C}D_{t}^{\beta }u_{\beta}(x,t) = \sigma(x)\dfrac{\partial^{2}u_{\beta
}(x,t)}{\partial x^{2}} + \mu(x)\dfrac{\partial u_{\beta}(x,t)}{\partial x}.
\end{equation*}
By Lemma \ref{lem3}, it can be written as 
\begin{align*}
\int_{0}^{\infty}\dfrac{\partial u(x,w)}{\partial w}g_{\beta }(w;s)dw &=
\sigma(x)\dfrac{\partial^{2}}{\partial x^{2}}\int_{0}^{\infty}
u(x,w)g_{\beta}(w;s)dw \\
&+ \mu(x)\dfrac{\partial}{\partial x}\int_{0}^{\infty}
u(x,w)g_{\beta}(w;s)dw.
\end{align*}
Taking a $\alpha $-Wright type transformation on both sides of the equation
we obtain, 
\begin{align*}
\int_{0}^{\infty}\int_{0}^{\infty}\dfrac{\partial u(x,w)}{\partial w}%
g_{\beta }(w;s)dw g_{\alpha }(s;t)ds &= \int_{0}^{\infty} \sigma(x) \dfrac{%
\partial^{2}}{\partial x^{2}} \int_{0}^{\infty} u(x,w)g_{\beta}(w;s)dw
g_{\alpha}(s;t)ds \\
&+ \int_{0}^{\infty}\mu(x)\dfrac{\partial}{\partial x}\int_{0}^{%
\infty}u(x,w)g_{\beta }(w;s)dw g_{\alpha}(s;t)ds
\end{align*}
and using \eqref{eq:CK} we obtain 
\begin{align*}
\int_{0}^{\infty}\dfrac{\partial u(x,w)}{\partial w} g_{\beta \alpha}(w;t)dw
&= \sigma(x)\dfrac{\partial^{2}}{\partial x^{2}}\int_{0}^{\infty}u(x,w)
g_{\beta \alpha}(w;t)dw \\
&+ \mu(x)\dfrac{\partial}{\partial x}\int_{0}^{\infty}u(x,w)g_{\beta%
\alpha}(w;t)dw,
\end{align*}
which completes the proof.
\end{proof}
\section{Numerical Simulations}\label{numerical}
We use Monte Carlo integration to simulate the solutions of factional differential equations and partial differential equations given their solutions. As expected, simulations of heavy tail distributions require the use of a large amount of random numbers for adequate coverage. We first present the lemma by M. Kanter~\cite{kanter1975} to generate random numbers distributed as the L\'evy $\alpha$-stable distribution with stability index $0< \alpha \leq 2$, $L_{\alpha}^{\theta}(x)$, defined by equation \eqref{levy-alpha}. See also \cite{Cahoy2012}.
\begin{lemma}[Lemma 4.1~\cite{kanter1975}]
Let $\alpha\in(0,1)$ and let $L_{\alpha}^{-\alpha}(x)$ as defined in \eqref{levy-alpha}. Then for $x\geq0$
\begin{equation}
L_{\alpha}^{-\alpha}(x) = \dfrac{1}{\pi}\left(\dfrac{\alpha}{1-\alpha}\right)\left(\dfrac{1}{x}\right)^{(1-\alpha)^{-1}} \int_{0}^{\pi} a(\varphi)\exp\left(-\left(\dfrac{1}{x}\right)^{\alpha/(1-\alpha)}\right) d\varphi
\end{equation}
where
\begin{equation}\label{eq-a}
a(\varphi) = \left(\dfrac{\sin(\alpha\varphi)}{\sin(\varphi)}\right)^{(1-\alpha)^{-1}}\left(\dfrac{\sin((1-\alpha)\varphi)}{\sin(\alpha\varphi)}\right).
\end{equation}
\end{lemma}
\begin{corollary}[Corollary 4.1~\cite{kanter1975}]
Let $U_1$ and $U_2$ be independent random variables where $U_1$ is uniformly distributed on $[0,1]$, and $U_2$ is uniformly distributed on $[0,\pi]$. Then for $\alpha\in(0,1)$, $L_{\alpha}^{-\alpha}(x)$ is the density of $(-a(U_2)/\log(U_1))^{(1-\alpha)/\alpha}$ where $a$ is given by equation \eqref{eq-a}.
\end{corollary}
\begin{corollary}
If $X\sim {L}_{\beta}^{-\beta}(\cdot)$, then $g_{\beta}(\cdot,t)$ is the probability density function of $Y=\dfrac{t^{\beta}}{X^{\beta}}$.
\begin{proof}
For $t>0$, the probability density function of $Y$ is given by
\begin{align}
\label{eqn:levgb}
\begin{split}
f_{Y}(y) &= {L}_{\beta}^{-\beta}\left(\dfrac{t}{y^{1/\beta}}\right)\left|\dfrac{dx}{dy}\right| \\
&= \dfrac{t}{\beta y}\dfrac{1}{y^{1/\beta}} {L}_{\beta}^{-\beta}\left(\dfrac{t}{y^{1/\beta}}\right) \\
&= \dfrac{t}{\beta y} {L}_{\beta}^{-\beta}(t,y) \\
&= g_{\beta}(y,t).
\end{split}
\end{align}
\end{proof}
\end{corollary}
\begin{corollary}
If $X\sim {L}_{\alpha}^{-\alpha}(\cdot)$, then ${L}_{\alpha}^{-\alpha}(\cdot,t)$ is the probability density function of $Y=Xt^{1/\alpha}$.
\begin{proof}
Similar proof as above.
\end{proof}
\end{corollary}
${L}_{\alpha}^{-\alpha}(\cdot,t)$ can be seen as the density of a $\alpha$ stable subordinator, here the skewness parameter equals -$\alpha$. By using the above corollaries, a powerful computational approach is performed using $g_{\beta}$ instead of simulating the Wright function. The codes are done on Python using numba and matplotlib libraries. 
\begin{example}[Fractional growth/decay models]
The solution of the Fractional differential equation
\begin{equation}\label{eq_mit}
\begin{aligned}
\left\{
\begin{split}
D_{C}^{\beta} y(t) &=  -y(t) \\
y(0) &= 1
\end{split}\qquad
\right.
\end{aligned}
\end{equation}
has a solution given by $y(t) = \int_{0}^{\infty}  e^{ -x} g_{\beta}(x;t)dx =  E_{\beta}(- t^{\beta})$.
\begin{figure}[htp]
\centering
\includegraphics[scale=0.35]{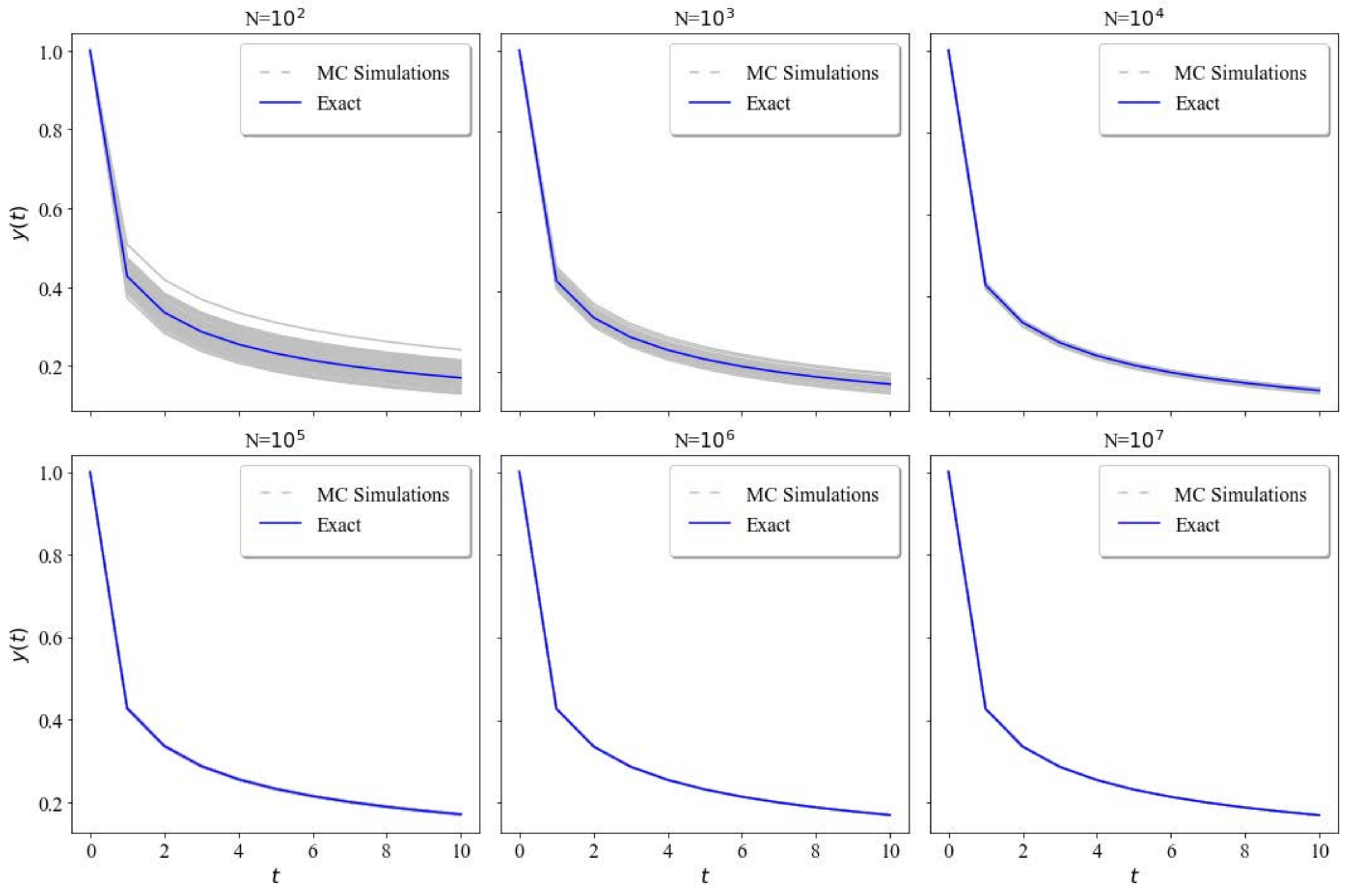}
\caption{The Monte Carlo integration repeated 100 times of simulations of the solution of equation \eqref{eq_mit} for each $\beta=0.5$ compared to the exact solution.}
\end{figure}

As a conclusion, the numerical solution converges to the actual solution when the random number values of $g_{\beta}$  are considerably high. Convergence is, however, good starting from $10^{4}$ randomly generated values from $g_{\beta}$.
\end{example}
\begin{example}
Let's consider the standard fractional Fokker-Planck equation of the form%
\begin{eqnarray}
_{C}D_{t}^{\beta }p_{\beta }(x,t) &=&\frac{\partial ^{2}p_{\beta }}{\partial
x^{2}}+x\frac{\partial p_{\beta }}{\partial x}+p_{\beta } \\
p(x,0) &=&f(x).
\end{eqnarray}
 We obtained a solution in the form
\begin{equation}\label{eq-FP}
p_{\beta }(x,t)=\int_{0}^{\infty }\int_{0}^{\infty }
\frac{\exp \left[ -\frac{\left( x-e^{-s}y\right) ^{2}}{2\left(
1-e^{-2s}\right) }\right] }{\sqrt{2\pi \left( 1-e^{-2s}\right) }}g_{\beta}(s;t) ds f(y)dy.
\end{equation}
See Figure \ref{Fig2} for numerical simulation of the solution and see Figure \ref{Fig3} for the mean absolute error for $N=10^6$ with respect to $N=10^7$.
\begin{figure}[htp]
\centering
\begin{tabular}{ccc}
\includegraphics[width=5cm]{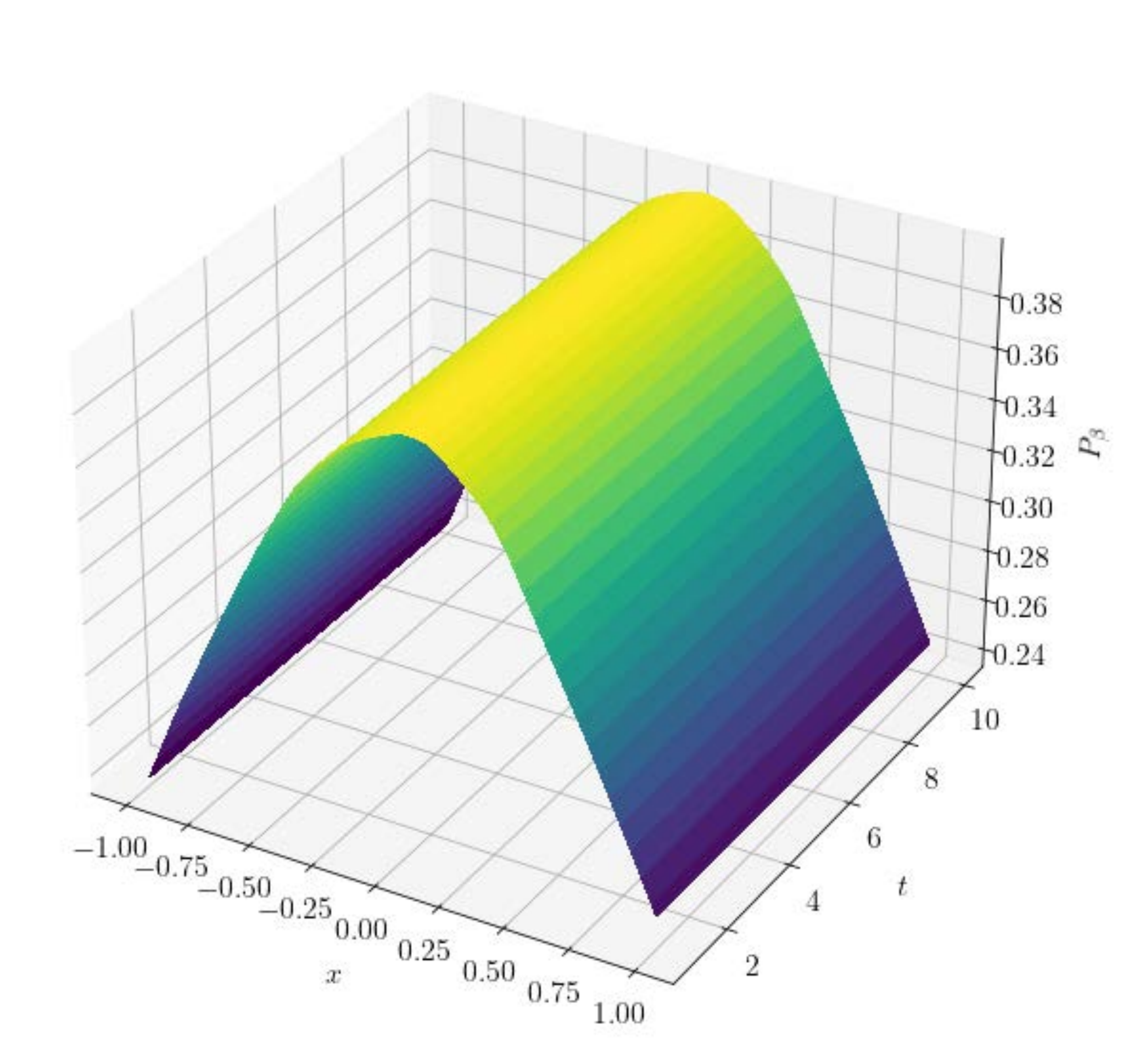} & \includegraphics[width=5cm]{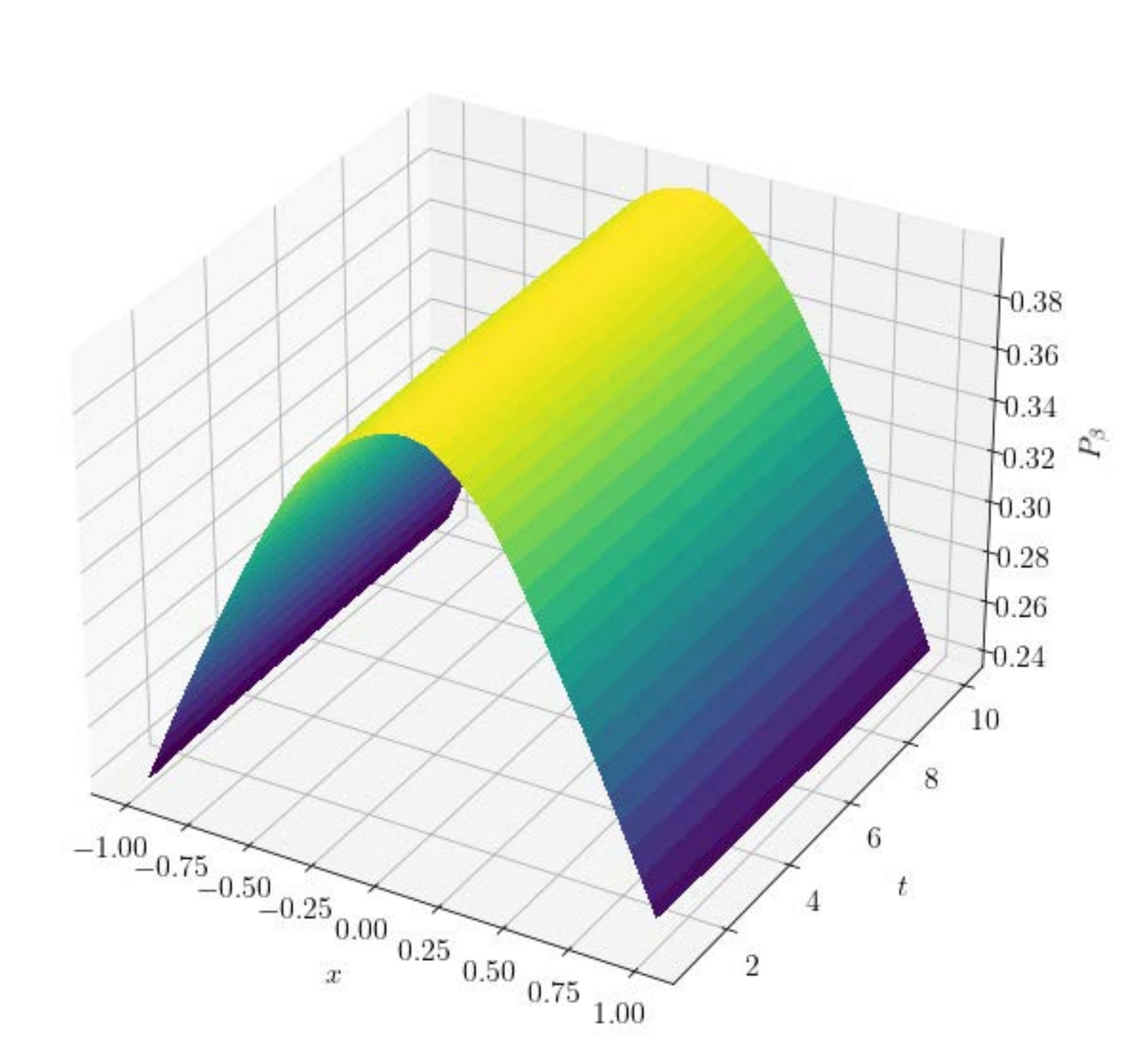} & \includegraphics[width=5cm]{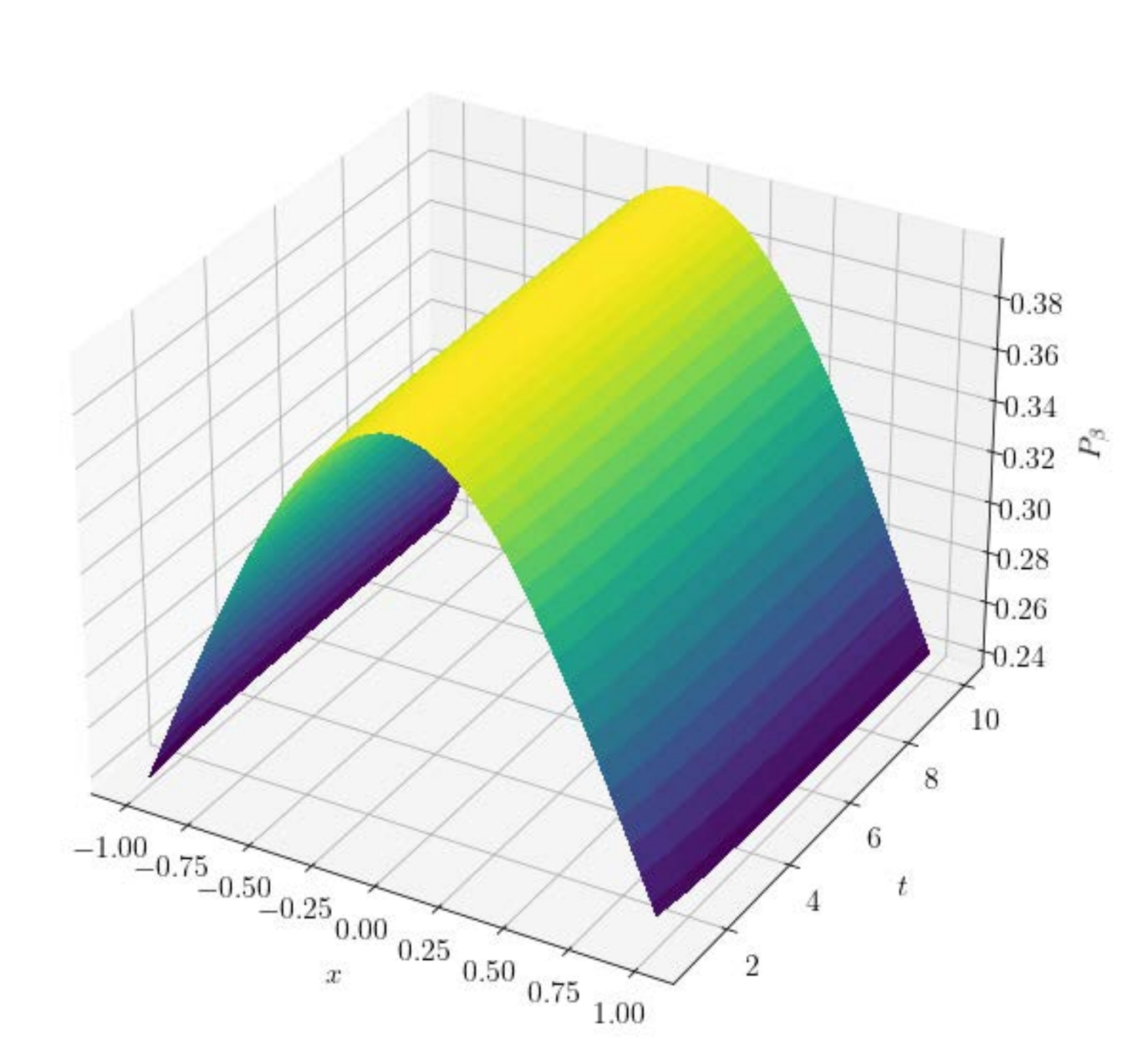} \\
(a)  & (b) & (c)
\end{tabular}
\caption{Fokker Plank Monte Carlo integration simulation of solution in equation \eqref{eq-FP} using a number of $10^7$ random values for $\beta=0.1$, $0.5$ and $0.9$, respectively.}\label{Fig2}
\end{figure}
\begin{figure}[htp]
\centering
\begin{tabular}{ccc}
\includegraphics[width=5cm]{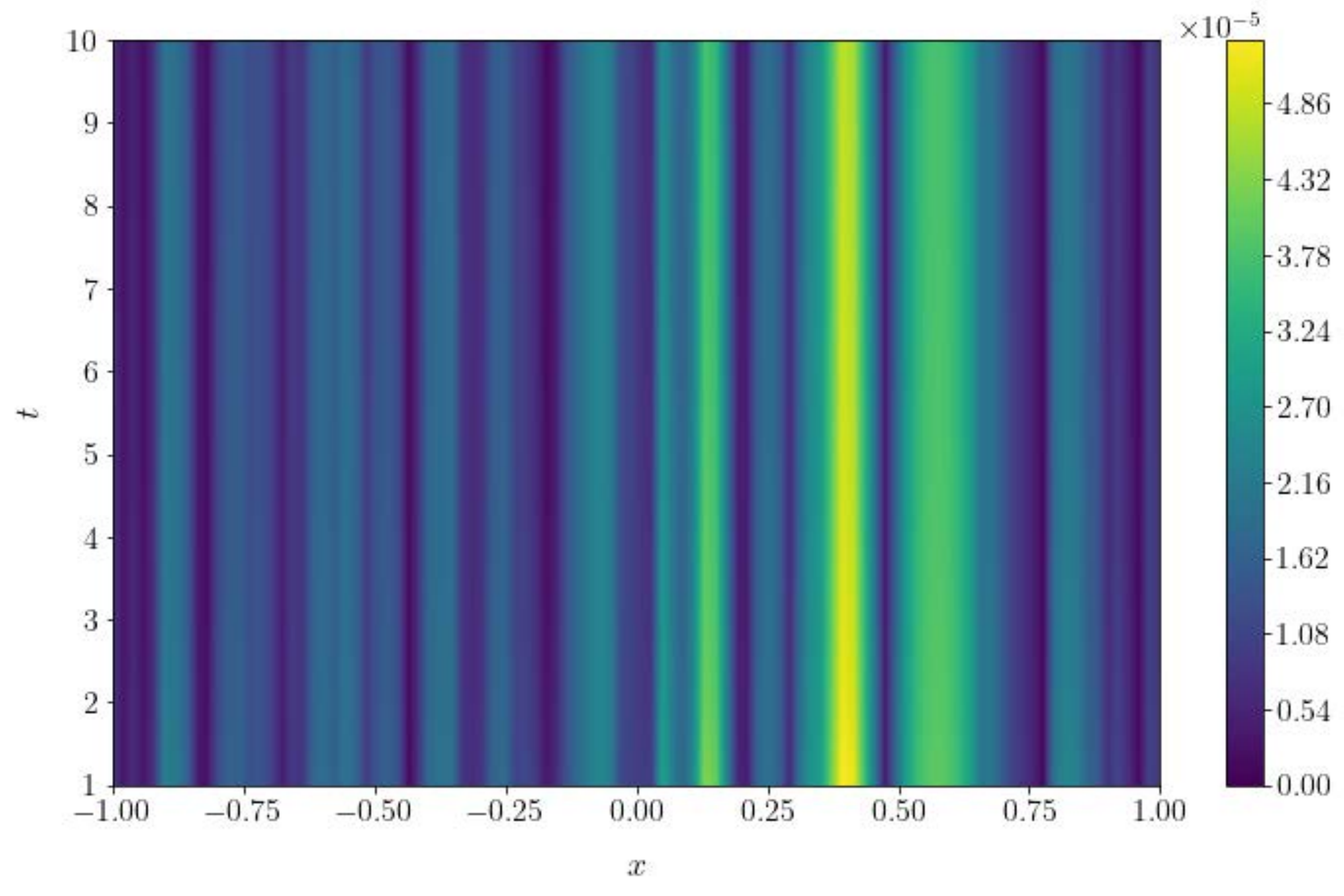} & \includegraphics[width=5cm]{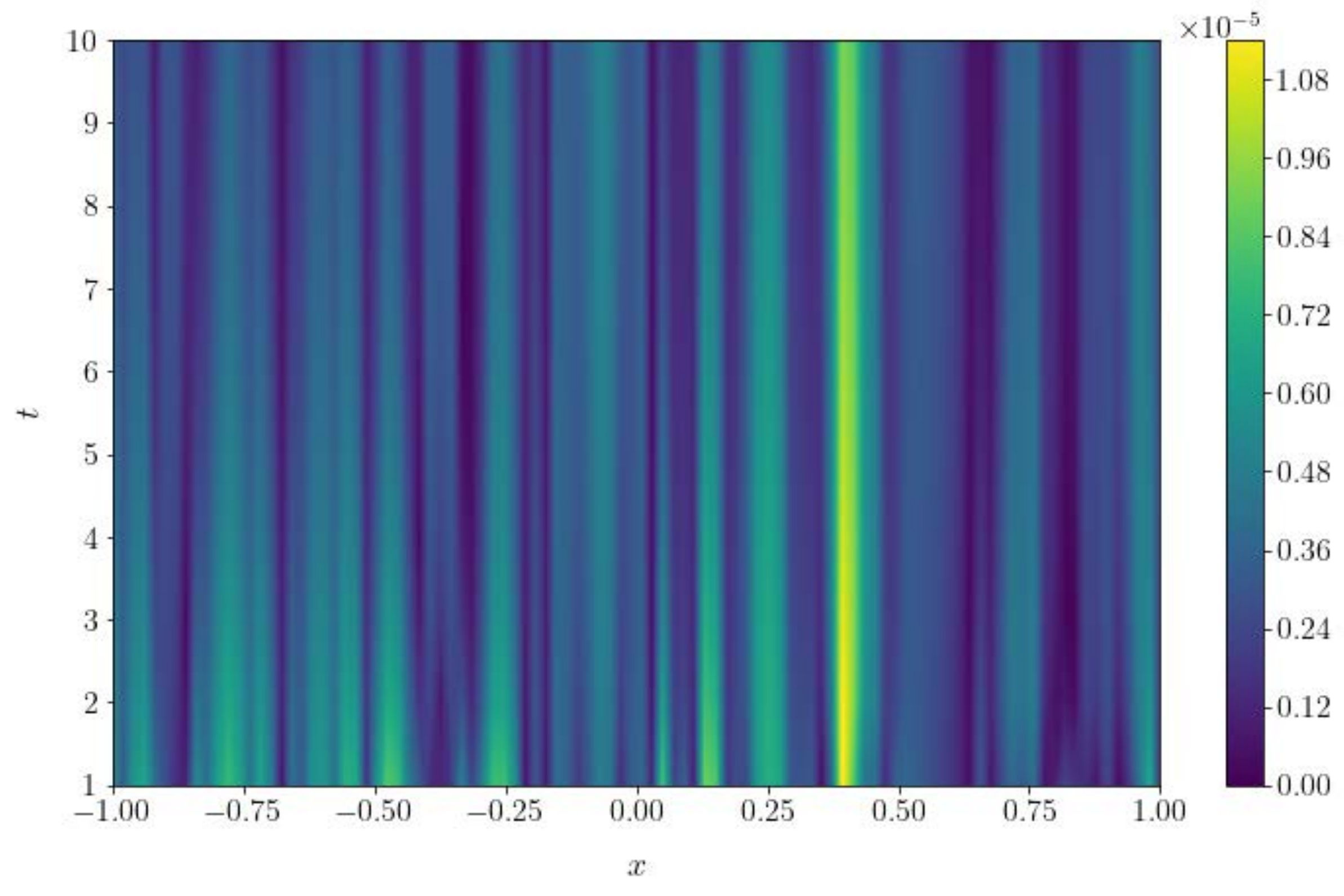} & \includegraphics[width=5cm]{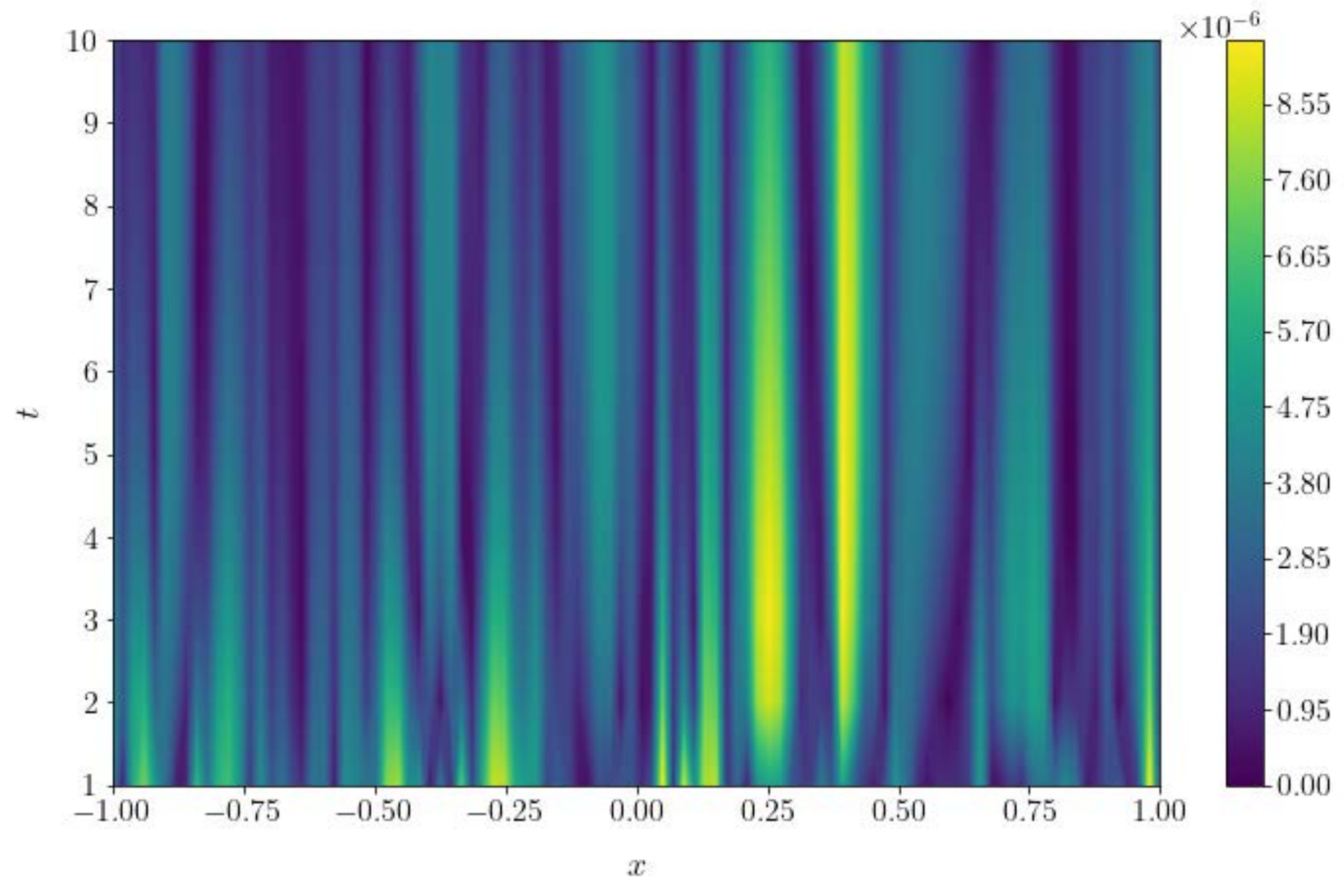} \\ 
(a)  & (b) & (c)  
\end{tabular}
\caption{Mean error in Fokker Plank Monte Carlo integration simulation of solution in equation \eqref{eq-FP} over time $t\in[1,10]$. The mean absolute error between using a number of $10^7$ and $10^6$ random values from $g_{\beta}$ for $\beta=0.1$, $0.5$ and $0.9$ respectively.}\label{Fig3}
\end{figure}
\end{example}
\begin{example}
The solution of the fractional heat/diffusion equation
\begin{equation}\label{eq-heat}
    u_{\alpha,\beta}(x,t)=\int_{0}^\infty \int_{0}^\infty L_2^0(x,\tau)L_{\alpha/2}^{-\alpha/2}(\tau,s)g_\beta(s,t) d\tau ds
\end{equation}
could be numerically computed for each $x$ and $t$ using a large number of simulations from $g_\beta(s,t)$, and for each $s$ simulate a large number of $\tau$ are simulated from $L_2^0(x,\tau)$ $L_{\alpha/2}^{-\alpha/2}(\tau,s)$. Then the values of $L_2^0(x,\tau)$ are averaged up over all values of $\tau$ and then over all values of $s$. See Figure \ref{Fig4}. 

A similar procedure could be done to simulate random numbers from $u_{\alpha,\beta}(x,t)$, but after simulating one $s$ from $g_\beta(s,t)$, and then one $\tau$ from $L_{\alpha/2}^{-\alpha/2}(\tau,s)$, then we simulate one $x$ from $L_2^0(x,\tau)$. 
\begin{figure}[htp]
\centering
\begin{tabular}{ccc}
\includegraphics[width=5cm]{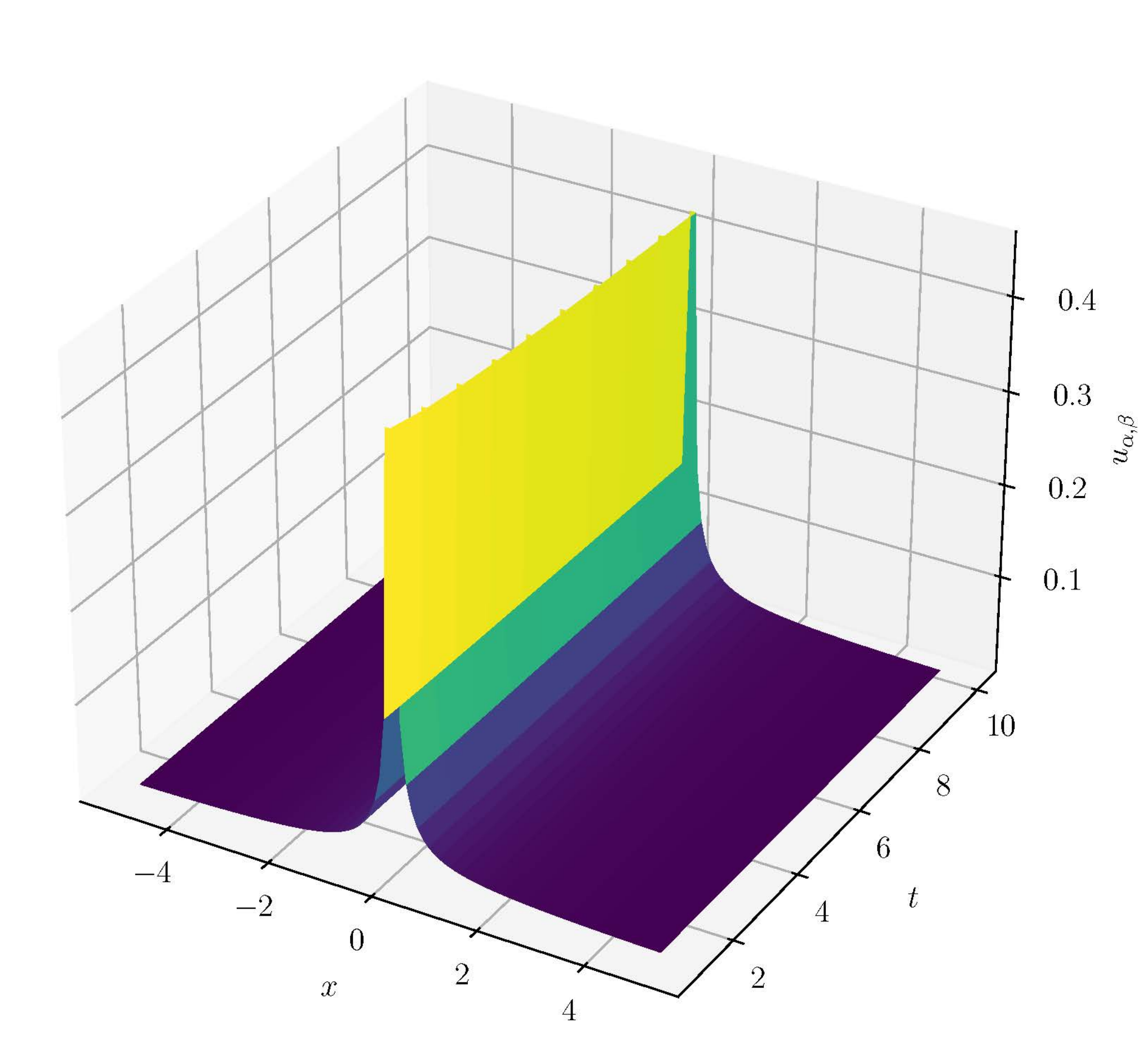} & \includegraphics[width=5cm]{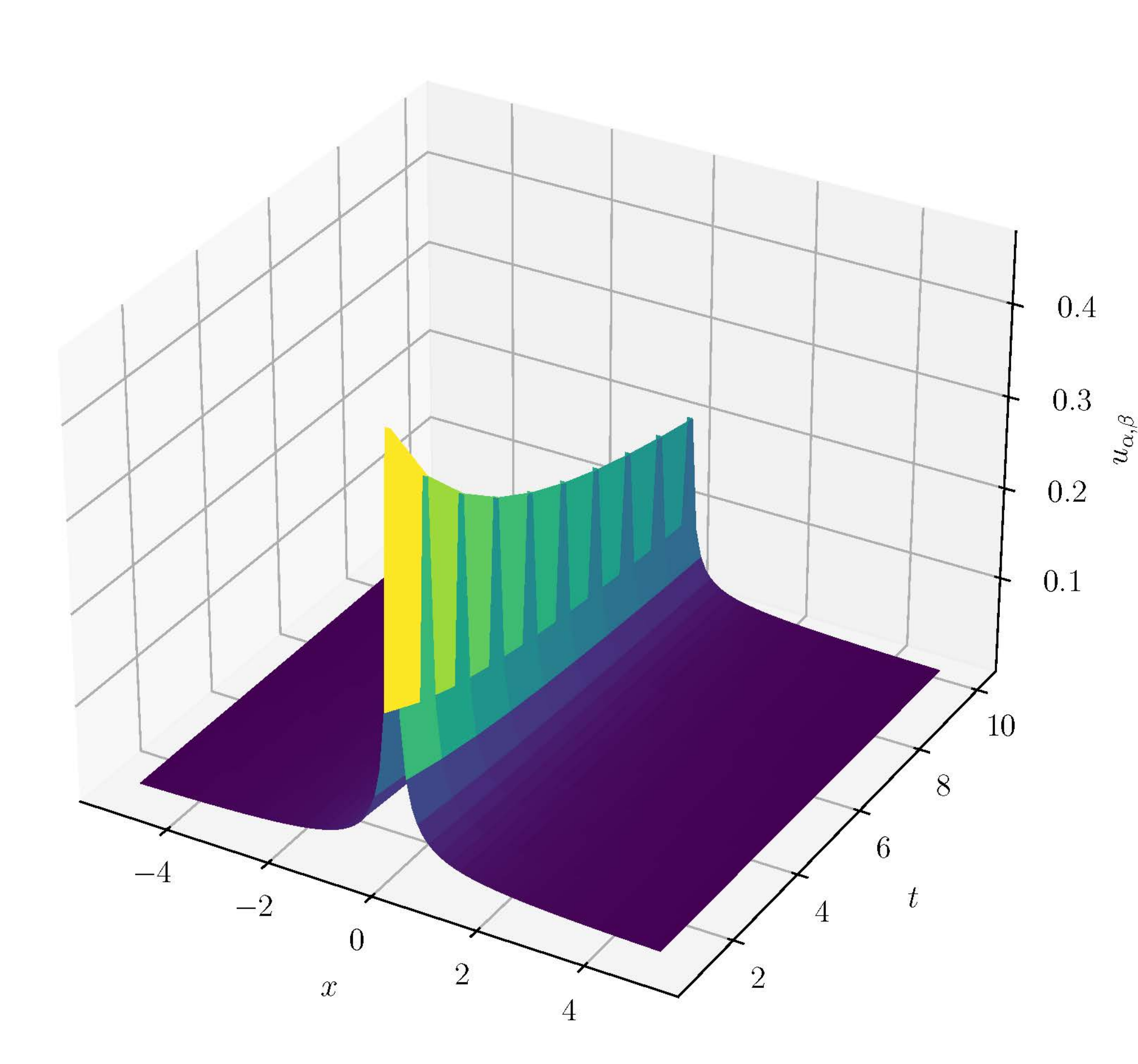} & \includegraphics[width=5cm]{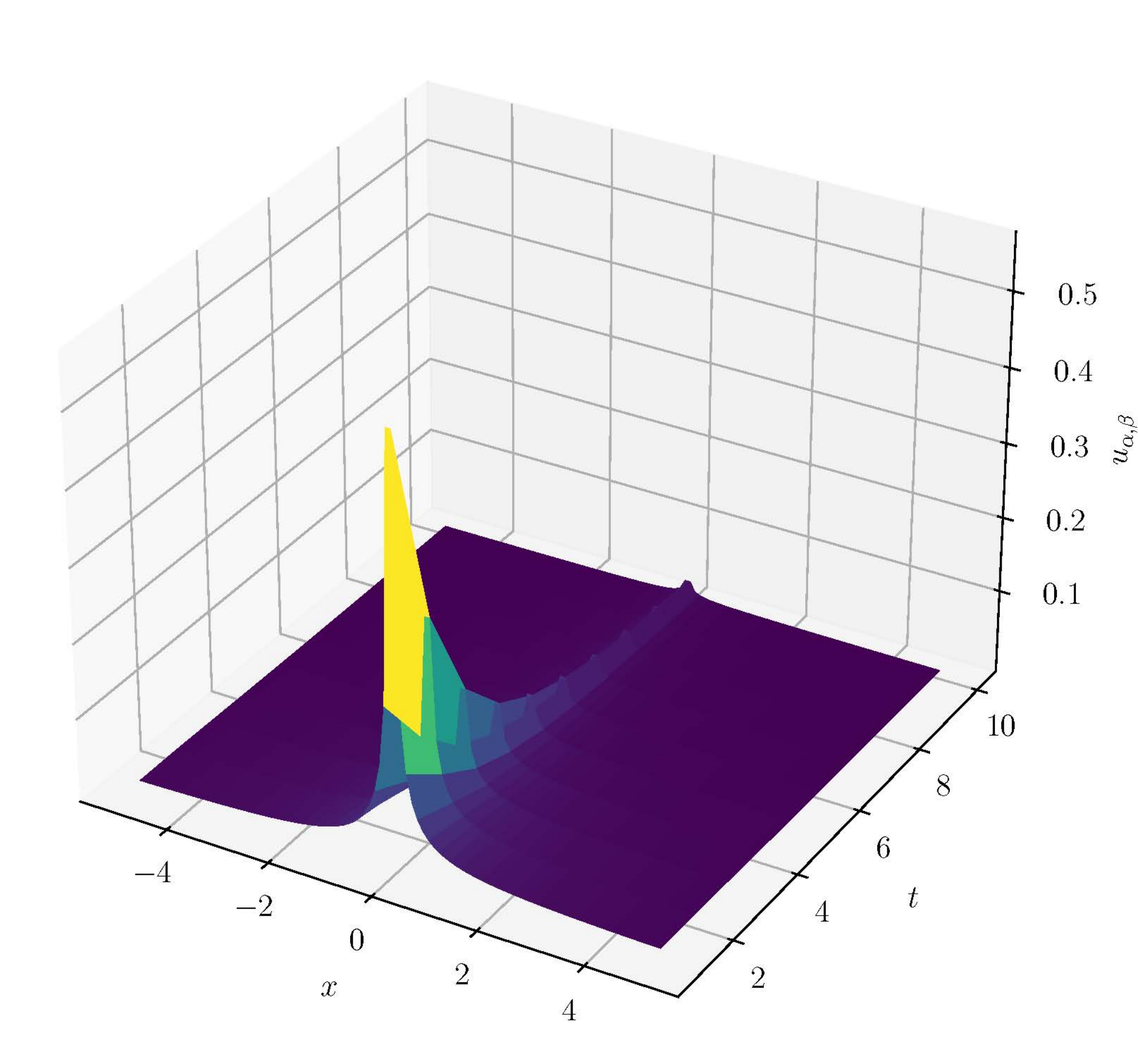} \\
(a) & (b) & (c) \\[6pt]
\includegraphics[width=5cm]{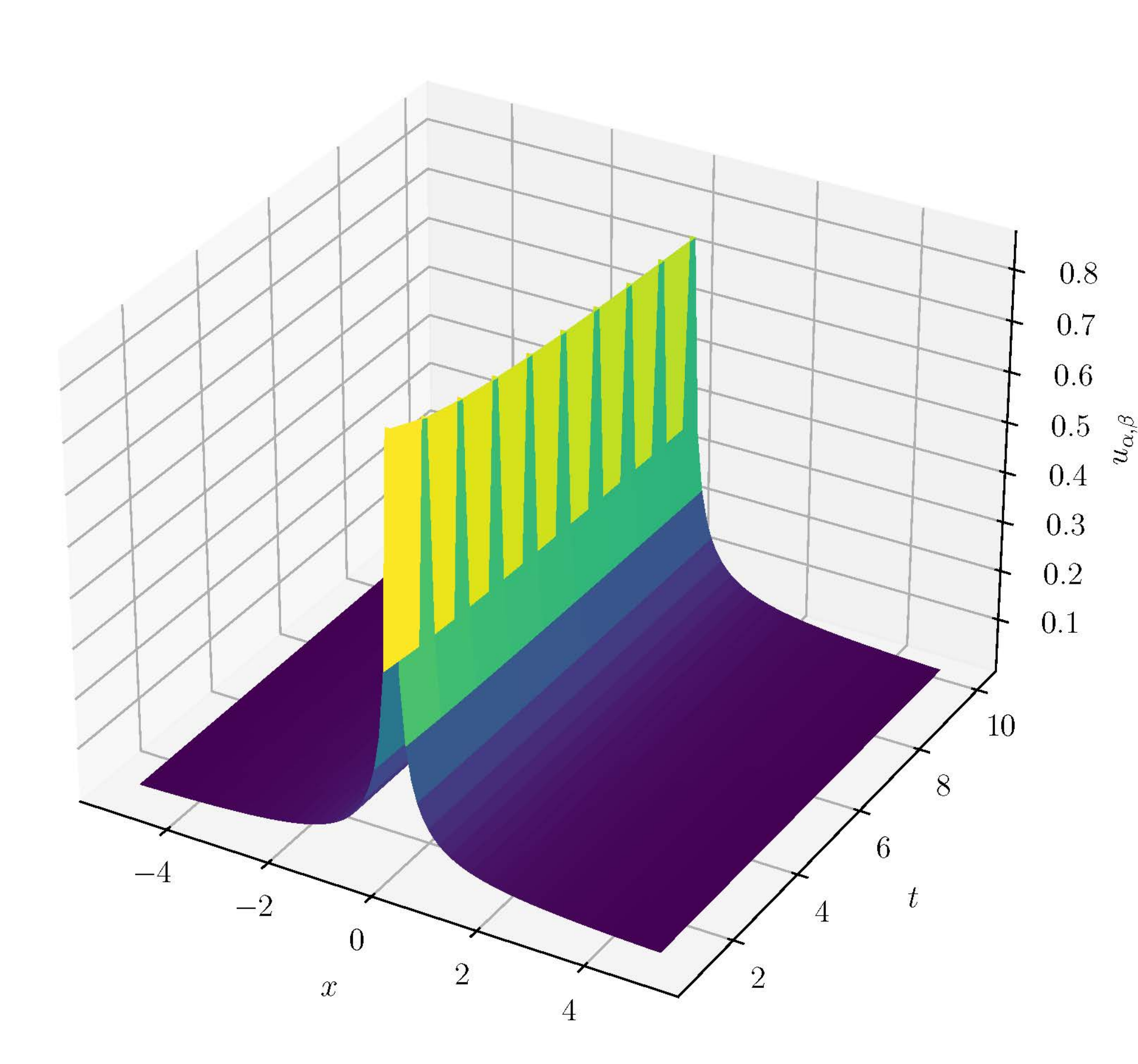} & \includegraphics[width=5cm]{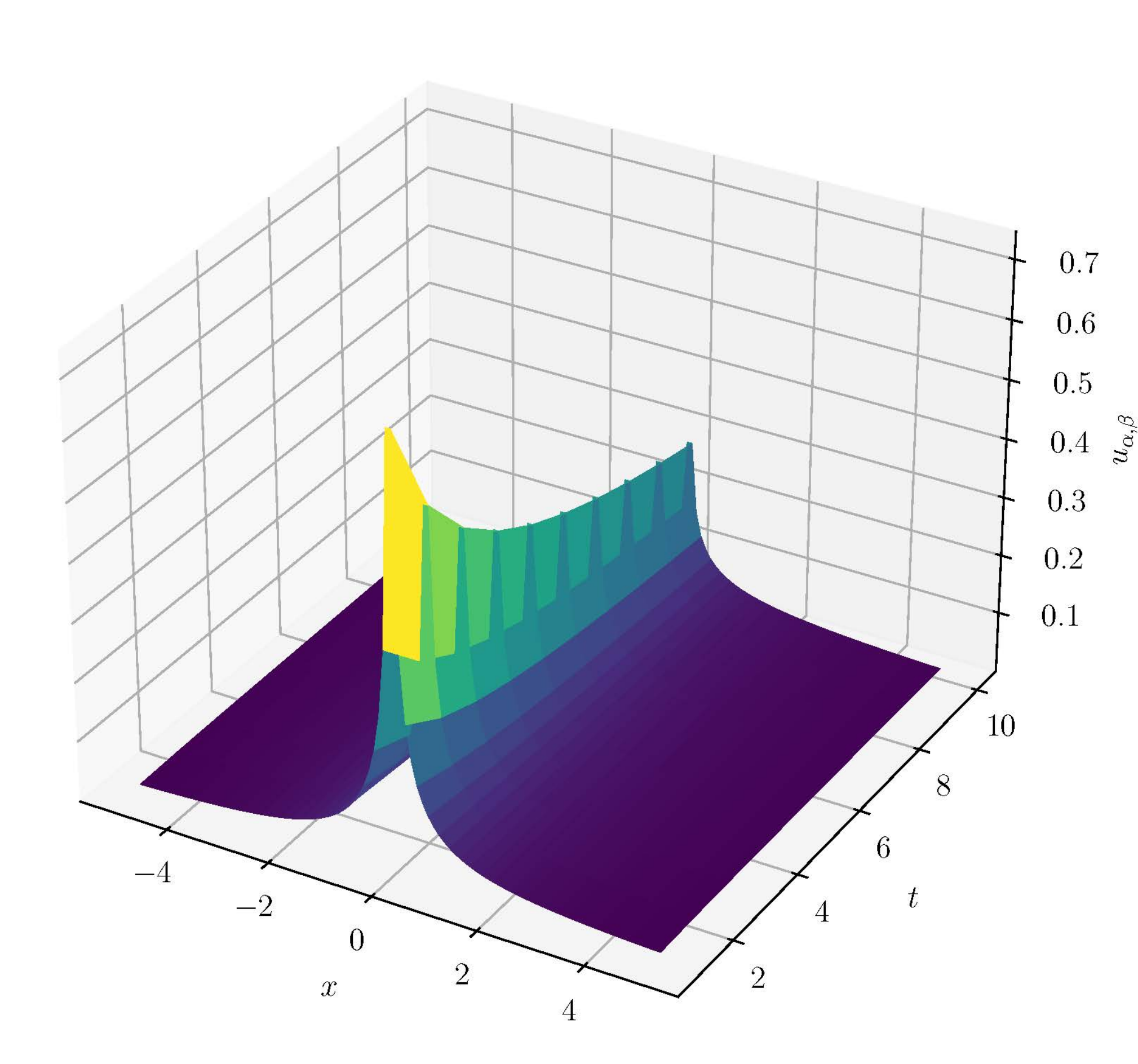} & \includegraphics[width=5cm]{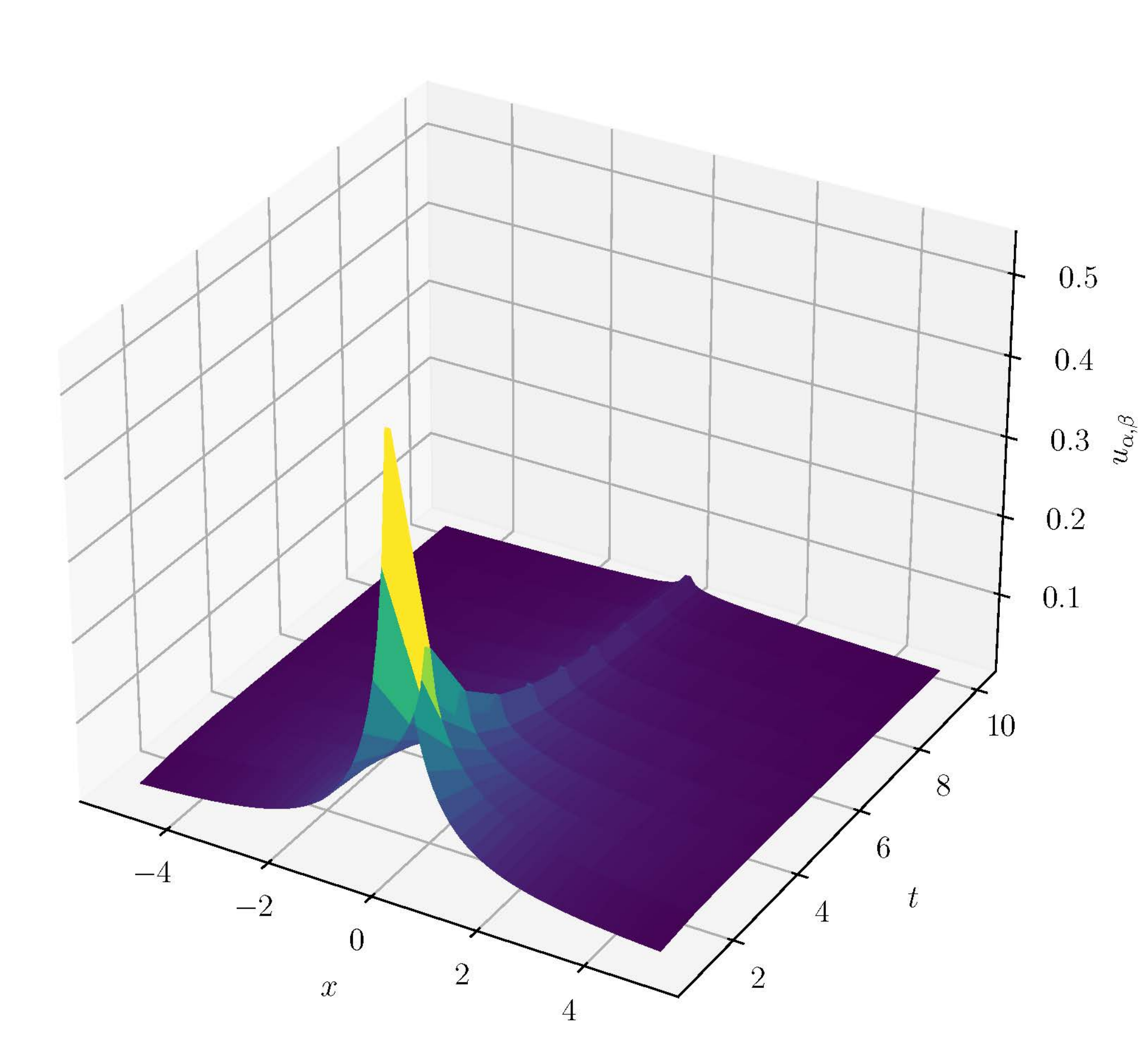} \\ 
(d)  & (e) & (f) \\[6pt]
\includegraphics[width=5cm]{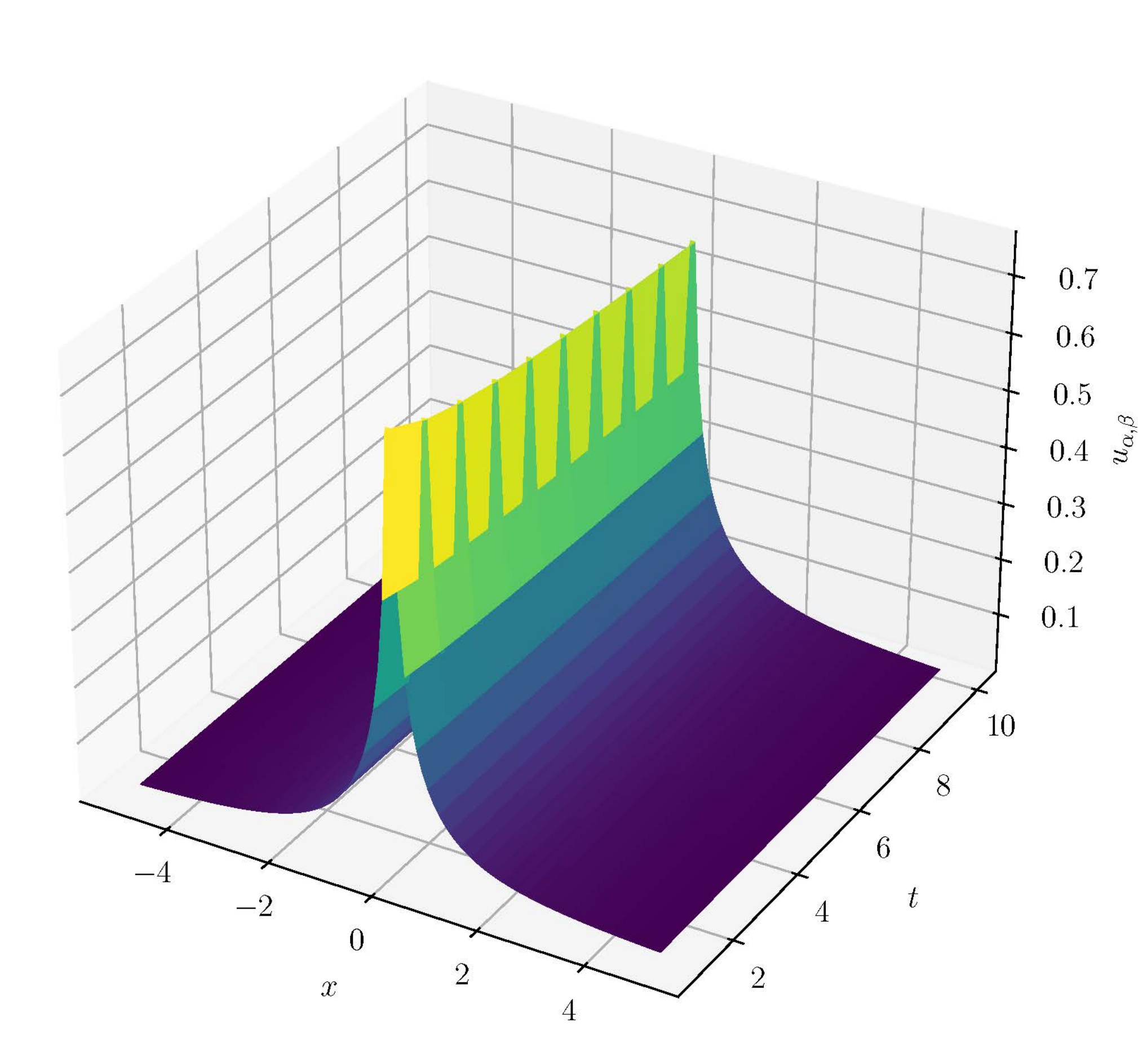} & \includegraphics[width=5cm]{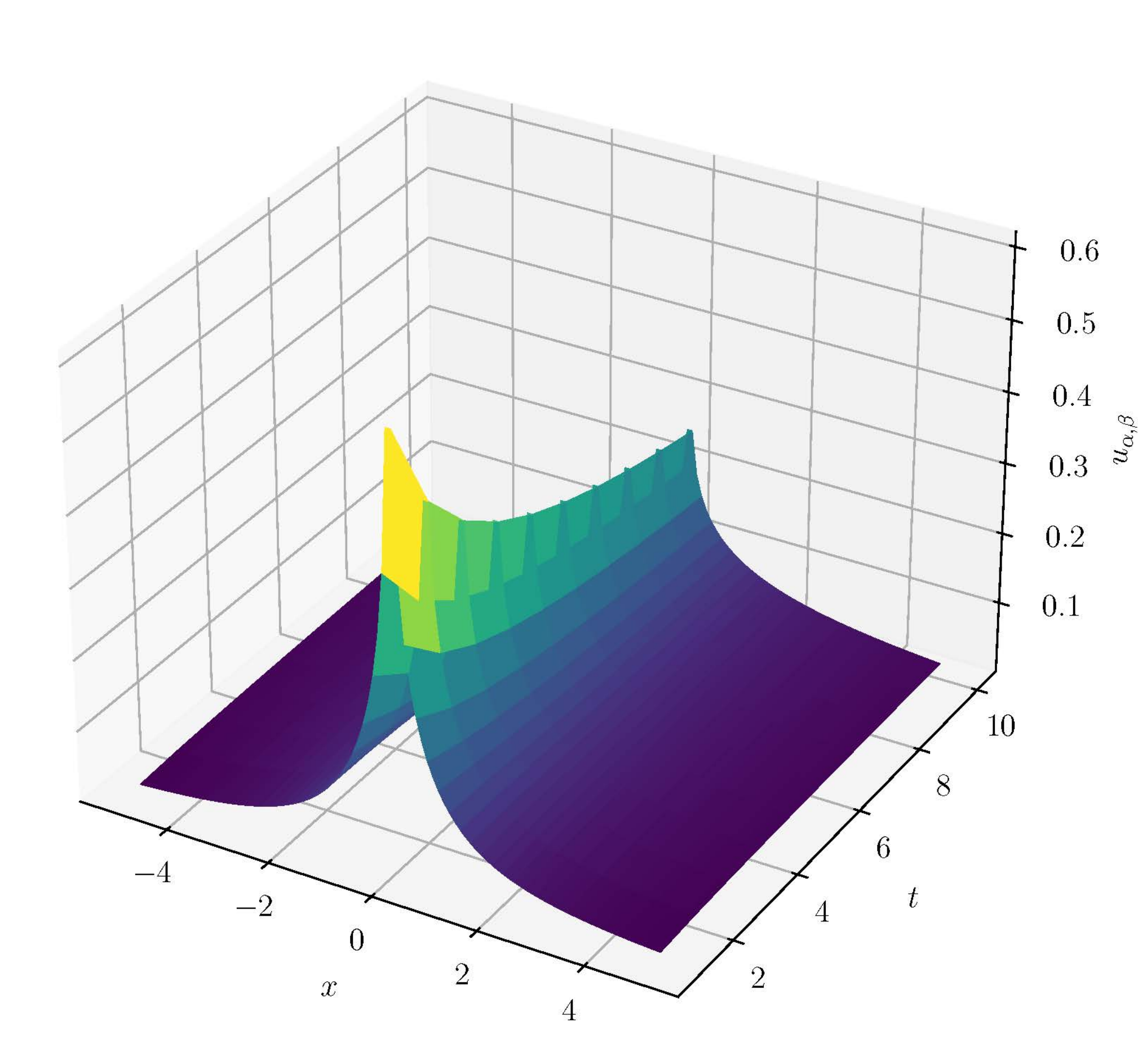} & \includegraphics[width=5cm]{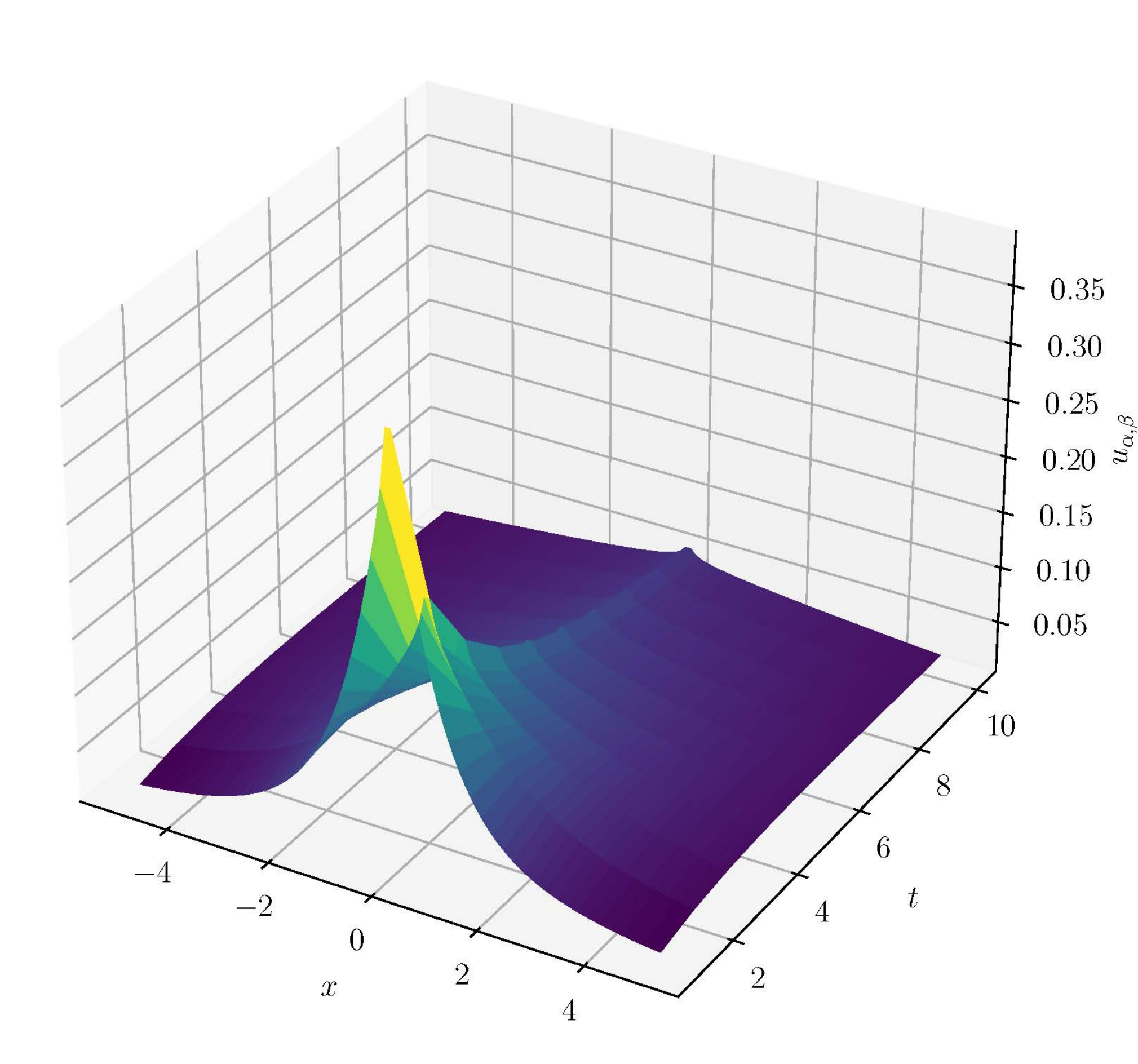} \\
(g)  & (h) & (i) \\
\end{tabular}
\caption{Simulation of the solution of fractional heat equation $u_{\alpha,\beta}$ of solution in equation \eqref{eq-heat} for $\beta=0.1$ in (a), (d) and (g), $\beta=0.5$ in (b), (e), and (h) and $\beta=0.9$ in (c), (f), and (i) and $\alpha=0.2$ in (a), (b) and (c), $0.6$ in (d), (e) and (f) and $1$ in (g), (h) and (i) for $(x,t)\in[-5,5]\times [1,10]$. The Monte Carlo integration was performed using a number of $10^7$ random values from $g_{\beta}$.}\label{Fig4}
\end{figure}
\end{example}
Linear ordinary differential equations and partial differential equations could be also solved numerically. Solutions then get interpolated at the values generated by $g_\beta(\cdot,t)$ and averaged at each value of $t$. 

\begin{example}
The solution of the fractional differential equation
\begin{equation}\label{last}
    D_C^{\beta+1} y(t)=- y(t)
\end{equation}
with $y(0)=0$ and $y'(0)=1$, could be numerically computed for each $t$ by numerically solving 
\begin{equation}\label{last2}
    D^2 z(x)=- z(x), 
\end{equation}
with $z(0)=0$, and $z'(0)=1$ at $N$ randomly generated $s$ values from $g_{\beta}(s,t)$ and then average the values up for each $t$. See Figure \ref{Fig5}. We used the Runge-Kutta method of hybrid order 4 and 5 from minimum generated $s$ to maximum generated $s$ and then interpolated the solution at the rest of the  randomly generated $s$ values. The solution is also given by
is given by 
\begin{equation*}
\sin_{\beta}\left( t\right)=\int_0^\infty \sin( x)g_{\beta}(x;t)dx=   t^{\beta}E_{2\beta,\beta+1}\left(-t^{2\beta}\right).
\end{equation*}

\begin{figure}[htp]
\centering

 \includegraphics[width=10cm]{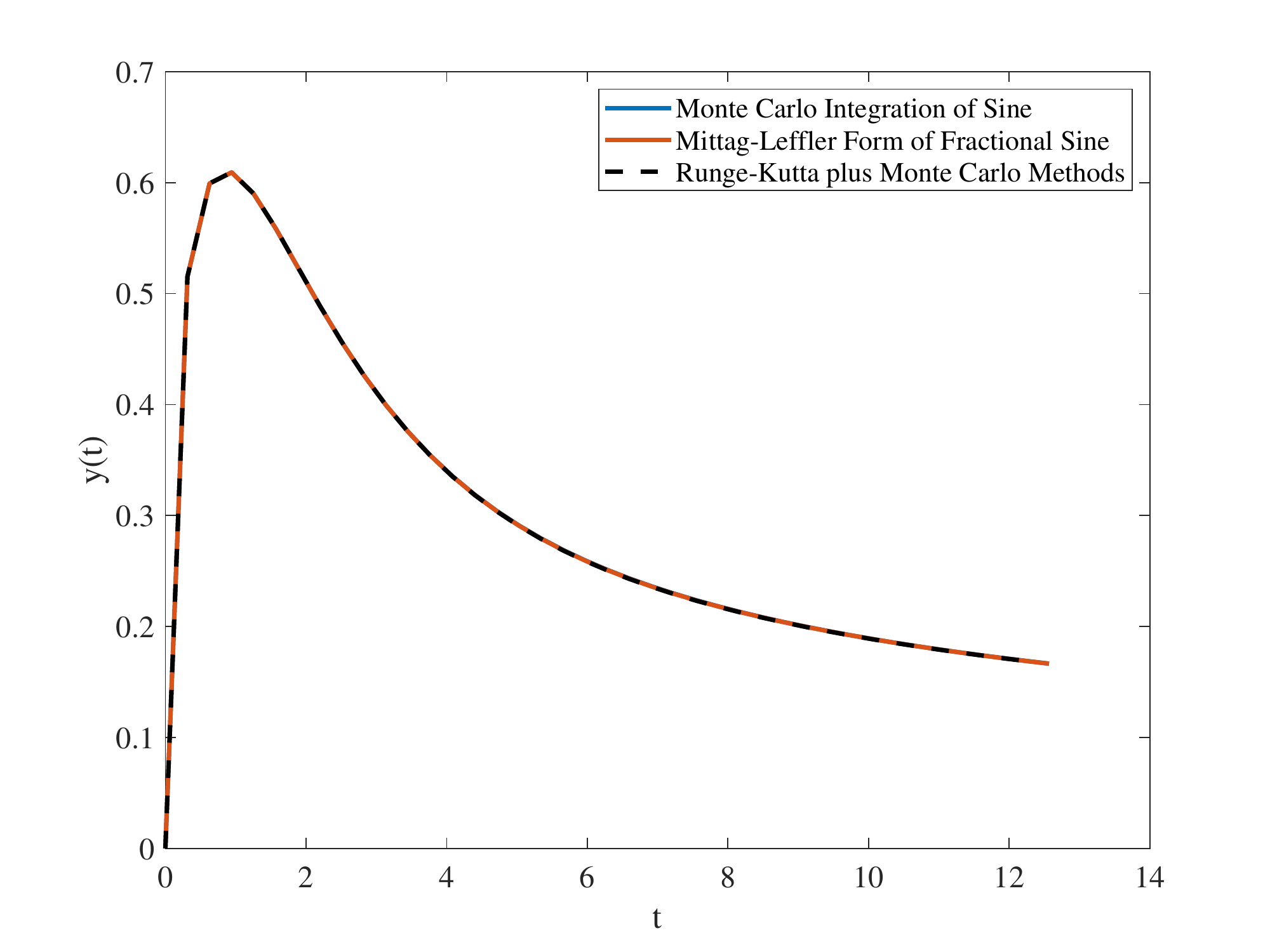} 
\caption{Monte Carlo simulation of the solution of fractional differential equation \eqref{last}, for $\beta=0.5$. The Monte Carlo integration was performed using a number of $N=10^7$ random values from $g_{\beta}$. The figures also show direct Monte Carlo integration of the sine function and the Mittag-Leffler function definition of the fractional sine function.}\label{Fig5}
\end{figure}
\end{example}

\section{Conclusion}
In this work, we have shown an alternative way to find the fundamental solutions for fractional partial differential equations (PDEs). Indeed, Riccati equation have allow us to study families of fractional diffusion PDEs with variable coefficients which allow explicit solutions. Those solutions connect Lie symmetries to fractional PDEs. We expect similar results for fracational dispersive equations. These results will appear in another work.

In our approach, we have taken advantage of Lie symmetries applied to fractional diffusion PDEs with variable coefficients. We predict that Feynman path integrals can play a similar role with fractional dispersive equations.

We conjecture that a general solution similar to that in equation \eqref{new} can be shown to hold for a larger class of fractional partial differential equations. A conjecture for which an interesting consequence that the L\'evy flight could be a result of the wide expanse allowed for the normal diffusion to make a jump over. 

Monte-Carlo integration of solutions of ordinary differential equations and partial differential equations with respect to heavy-tailed distributions is a new approach that can prove valuable for other general fractional equations. Evaluating such solutions take small amount of time thanks to the availability of fast computing devices. A general formula like equation \eqref{new} can make numerical solutions easier.





\end{document}